\documentclass[12pt,a4paper]{amsart}
\usepackage{latexsym,amsfonts,amsmath,amssymb,amsthm,rotating,txfonts}
\usepackage{epic}
\usepackage{diagrams}
\input xy
\xyoption{all}

\usepackage{amscd,graphics}

\newcommand{\Proj}{{\rm Proj}}
\newcommand{\Spec}{{\rm Spec}}

\newcommand{\nc}{\newcommand}
\nc{\bla}{\phantom{bbbbb}}

\newcommand{\Sym}{ \,{\rm Sym} \,}

\newcommand{\beq}{\begin{equation}}
\newcommand{\eeq}{\end{equation}}
\newcommand{\barr}{\begin{array}}
\newcommand{\earr}{\end{array}}
\newcommand{\beqar}{\begin{eqnarray}}
\newcommand{\eeqar}{\end{eqnarray}}
\newtheorem{theorem}{Theorem}[section]
\newtheorem{corollary}[theorem]{Corollary}
\newtheorem{lemma}[theorem]{Lemma}
\newtheorem{prop}[theorem]{Proposition}
\newtheorem{definition}[theorem]{Definition}
\newtheorem{remark}[theorem]{Remark}

\newtheorem{exit}[theorem]{Example}

\newenvironment{rem}{\begin{remark}\rm}{\end{remark}}
\newenvironment{ex}{\begin{exit}\rm}{\end{exit}}
\newenvironment{defn}{\begin{definition}\rm}{\end{definition}}

\newcommand{\CC}{{\mathbb C }}
\nc{\FF}{ {\mathbb F} }
\nc{\HH}{ {\mathbb H} }

\newcommand{\PP}{ {\mathbb P } }

\newcommand{\UU}{{\mathbb U }}
\newcommand{\GG}{{\mathbb G }}

\newcommand{\curv}{\mathcal{C}}

\newcommand{\cale}{\mathcal{E}}

\newcommand{\calo}{\mathcal{O}}

\newcommand{\lie}[1]{\operatorname{\mathfrak{#1}}}
\newcommand{\kf}{\lie k}

\newcommand{\tf}{\lie t}

\newcommand{\impl}{\textup{impl}}

\newcommand{\symp}{{ /\!/ }} 

\DeclareMathOperator{\Stab}{Stab}

\nc{\conv}{{\rm Conv}}
\nc{\umax}{{U_{\max}}}
\newcommand{\weight}{\omega}
\newcommand{\weightlie}{\hat{\omega}}
\newcommand{\reg}{\mathrm{reg}}

\newcommand{\Lie}{\mathrm{Lie}}

\newcommand{\symk}{\sym^{\mathbf{\weight}}\CC^n}
\newcommand{\wsymk}[1]{\wedge^{#1}(\symk)}
\newcommand{\grass}{\mathrm{Grass}}

\newcommand{\bbg}{\mathbb{G}}

\newcommand{\bk}{\mathbf{\omega}}
\newcommand{\bi}{\mathbf{i}}

\newcommand{\bv}{\mathbf{v}}

\newcommand{\zdis}{\mathfrak{p}}

\newcommand{\ba}{\mathbf{\alpha}}

\newcommand{\bars}{\overline{s}}
\newcommand{\barss}{\overline{ss}}
\newcommand{\jetreg}[2]{J_{n}^{\mathrm{reg}}({#1},{#2})}

\newcommand{\lieu}{{\mathfrak u}}
\newcommand{\hU}{\hat{U}}

\newcommand{\xg}{X/\!/G}

\newcommand{\GL}{\mathrm{GL}}
\newcommand{\SL}{\mathrm{SL}}
\newcommand{\SU}{\mathrm{SU}}
\newcommand{\sym}{\mathrm{Sym}}

\newcommand{\symkd}{\sym^{\mathbf{\weight}}_{\Delta}\CC^n}
\nc{\lieq}{{\mathfrak q}}
\nc{\liez}{{\mathfrak z}}
\nc{\lieqs}{{\lieq}^*}
\nc{\lieg}{{\mathfrak g}}
\nc{\liegs}{{\lieg}^*}
\nc{\liep}{{\mathfrak p}}
\nc{\lieps}{{\liep}^*}

\def\a{\alpha}
\def\b{\beta}

\def\s{\sigma}

\title{Graded unipotent groups and Grosshans theory}
\author{Gergely B\'erczi and Frances Kirwan}
\address{University of Oxford, Mathematical Institute, Andrew Wiles Building, Oxford OX2 6GG, UK}

\thanks{Early work on this project was supported by the Engineering and Physical Sciences 
Research Council [grant numbers   GR/T016170/1,EP/G000174/1].}

\begin{document}

\begin{abstract}
Let $U$ be a unipotent group which is graded in the sense that it has an extension $H$ by the multiplicative group of the complex numbers such that all the weights of the adjoint action on the Lie algebra of $U$ are strictly positive. We study embeddings of $H$ in a general linear group $G$ which possess Grosshans-like properties. More precisely, suppose $H$ acts on a projective variety $X$ and its action extends to an action of $G$ which is linear with respect to an ample line bundle on $X$. Then, provided that we are willing to twist the linearisation of the action of $H$ by a suitable (rational) character of $H$, we find that the $H$-invariants form a finitely generated algebra and hence define a projective variety $X/\!/H$; moreover the natural morphism from the semistable locus in $X$ to $X/\!/H$ is surjective, and semistable points in $X$ are identified in $X/\!/H$ if and only if the closures of their $H$-orbits meet in the semistable locus. A similar result applies when we replace $X$ by its product with the projective line; this gives us a projective completion of a geometric quotient of a $U$-invariant open subset of $X$ by the action of the unipotent group $U$.
\end{abstract}

\maketitle

\section{Introduction}
Quotients of complex projective or affine varieties by linear actions of complex
reductive groups can be constructed and studied using
Mumford's geometric invariant theory (GIT)
 \cite{Dolg,GIT,New}.  Given a linear action on a complex projective
variety $X$ of a linear algebraic group $G$ which is {\em not}
reductive, the graded algebra of invariants is not necessarily finitely
generated, and even if it is finitely generated, so that there is a
GIT quotient $\xg$ given by the associated projective variety, the
geometry of this GIT quotient is hard to describe. When $G$ is
reductive then $\xg$ is the image of a surjective morphism
$\phi:X^{ss} \to \xg$ from an open subset $X^{ss}$ of $X$ (consisting
of the semistable points for the linear action), and $\phi(x) = \phi(y)$ if and
only if the closures of the $G$-orbits of $x$ and $y$ meet in $X^{ss}$.
When $G$ is not reductive $\phi:X^{ss} \to \xg$ can still be defined
in a natural way but it is not in general surjective, and indeed its image
is not in general an algebraic variety, even when the algebra
of invariants is finitely generated
\cite{DK}. 


One situation in which the algebra of invariants for a non-reductive linear algebraic group action on a projective variety $X$ with respect to an ample line bundle $L$ is guaranteed to be finitely generated is when the group $H$ is a Grosshans subgroup of a reductive group $G$ and the linear action of $H$ extends to $G$. Recall that a closed subgroup $H$ of $G$ is a Grosshans subgroup 
 \cite{Grosshans3,Grosshans}  if and only if the algebra of invariants $\mathcal{O}(G)^H$ is finitely generated and 
 $H$ is an observable subgroup of  $G$ in the sense that 
$$H = \{ g \in G : f(gx)=f(x) \mbox{ for all $x \in X$ and } f \in \mathcal{O}(G)^H \}$$
 (see $\S$3). 
 In this case $G/H$ is quasi-affine, and the finite generation of $\mathcal{O}(G)^H$ is equivalent to the existence of a finite-dimensional $G$-module $V$ and some $v \in V$ such that $H=G_v$ is the stabiliser of $v$ and $\dim(\overline{G\cdot v}\setminus G\cdot v)\le \dim(G\cdot v)-2$.  
Then we find that the non-reductive GIT quotient $X/\!/H$ (the projective variety associated to the algebra of invariants for the linear action of $H$) is given by the classical GIT quotient
$$ X /\!/ H = (X \times \Spec(\mathcal{O}(G)^H))/\!/G$$
of $X \times \Spec(\mathcal{O}(G)^H)$ by the diagonal action of $G$. Here $ \Spec(\mathcal{O}(G)^H) \cong \overline{G\cdot v}$ is the canonical affine completion of the quasi-affine variety $G/H$. The fact that the embedding $$ G/H \to G/\!/H = \Spec(\mathcal{O}(G)^H) $$
is not in general an isomorphism means that the natural map $X^{ss} \to X/\!/H$ is not in general surjective and $X/\!/H$ cannot be described as $X^{ss}$ modulo an equivalence relation, in contrast to classical GIT.
 
 When $H$ is a unipotent subgroup of a reductive group $G$ then $G/H$ is quasi-affine. In general if $H$ is a closed subgroup of a reductive group $G$ then $G/H$ is quasi-projective but not necessarily quasi-affine, so that  Grosshans theory does not apply directly. In this paper we will study families of subgroups $H$ of reductive groups $G$, where $H$ is neither reductive nor unipotent, which possess a property related to the Grosshans property. That is, given any action of $H$ on a projective variety $X$ extending to an action of $G$ which is linear with respect to an ample line bundle on $X$, then {\it provided} that we are willing to twist the linearisation of the action of $H$ by a suitable (rational) character of $H$ we find that the $H$-invariants form a finitely generated algebra; moreover the natural morphism $\phi: X^{ss} \to X/\!/H$ is surjective and satisfies $\phi(x) = \phi(y)$ if and
only if the closures of the $H$-orbits of $x$ and $y$ meet in $X^{ss}$. This property is weaker than the Grosshans property in that we are only guaranteed finitely generated invariants when we twist the linearisation by suitable rational characters of $H$ (though of course unipotent groups have no non-trivial characters). However it is stronger in the sense that we obtain a surjective morphism from $X^{ss}$ to $X/\!/H$ and a geometric description of $X/\!/H$ as $X^{ss}$ modulo the equivalence relation given by $x \sim y$ if and only if the closures of the $H$-orbits of $x$ and $y$ meet in $X^{ss}$.


Our first motivation for this investigation was 
the reparametrisation group $\GG_n$ consisting of $n$-jets of
germs of biholomorphisms of $(\CC,0)$, which acts on the jet bundle $J_n(Y)$ over
a complex manifold $Y$ (for some positive integer $n$). The fibre of $J_n(Y)$ over $x\in Y$ is the space of $n$-jets of germs at the origin of 
holomorphic curves $f:(\CC,0) \to (Y,x)$, and polynomial functions on $J_n(Y)$ are algebraic differential operators $Q(f',\ldots,f^{(n)})$, called jet differentials. The reparametrisation group $\GG_n$ acts fibrewise on the bundle $E_n(Y)$ of jet differentials. This action has  played a central role in the history of hyperbolic varieties and the Kobayashi conjecture on the non-existence of holomorphic curves in compact complex manifolds of generic type (see for example  \cite{demailly,dmr,Merker1,Merker2}).

The reparametrisation group $\GG_n$ is the semi-direct product $\UU_n \rtimes \CC^*$ of its unipotent
radical $\UU_n$ with $\CC^*$. It is a subgroup of $\GL(n)$ with the 
upper triangular form  
$$
\GG_n \cong \left\{ 
\left(\begin{array}{ccccc}
\a_1 & \a_2 & \a_3 & \cdots  & \a_n \\
0        & \a_1^2 &  \cdots &  &  \\
0        & 0       & \a_1^3  & \cdots &  \\
\cdot    & \cdot   & \cdot    & \cdot &  \cdot \\
0 & 0 & 0 & \cdots  & \a_1^n 
\end{array} \right) : \a_1 \in \CC^*, \a_2,\ldots,\a_n \in \CC \right\}
$$
where the entries above the leading diagonal are polynomials in $\a_1, \ldots, \a_n$, and $\UU_n$ is the subgroup consisting of matrices of this form with $\a_1=1$.
Notice that if $n$ is odd then the embedding of $\GG_n$ in $\GL(n)$ can be modified by multiplying a matrix in $\GG_n$ with first row $(\a_1, \ldots, \a_n)$ by the scalar $(\a_1)^{-(n+1)/2}$. The  image of this modified embedding is a subgroup $\tilde{\GG}_n$ of $\SL(n)$; if $n$ is even then the corresponding subgroup $\tilde{\GG}_n$ of $\SL(n)$ is a double cover of ${\GG}_n$.

This paper studies more generally actions of groups $U$, $\hat{U}$ and $\tilde{U}$ of a similar  form to those of $\UU_n$, $\GG_n$ and $\tilde{\GG}_n$. Let $U$ be a unipotent subgroup of $\SL(n)$ with a semi-direct product
$$\hat{U} = U \rtimes \CC^*$$ where $\CC^*$ acts on the Lie algebra of
$U$ with all its weights strictly positive; we call such groups graded unipotent groups. We assume that $U$ and $\hat{U}$ are 
upper triangular subgroups of $\GL(n)$ which are \lq generated along the first row' in the
sense that there are integers $1 = \weight_1 < \weight_2 \leq \weight_3 \leq \cdots \leq \weight_n$
and polynomials $\zdis_{i,j}(\alpha_1,\ldots,\alpha_n)$ in $\alpha_1,\ldots,\alpha_n$
with complex coefficients for $1<i<j \leq n$ such that 
\begin{equation} \label{label1}
\hU=\left\{\left(\begin{array}{ccccc}\a_1 & \a_2 & \a_3 & \ldots & \a_n \\ 0 & \a_1^{\weight_2} & p_{2,3}(\ba) & \ldots & p_{2,n}(\ba) \\ 0 & 0 & \a_1^{\weight_3} & \ldots & p_{3,n}(\ba) \\ \cdot & \cdot & \cdot & \cdot &\cdot \\ 0 & 0 & 0 & 0 & \a_1^{\weight_n}  \end{array}\right)
: \ba =(\a_1,\ldots, \a_n) \in \CC^* \times \CC^{n-1} \right\}
\end{equation}
and $U$ is the unipotent radical of $\hU$; that is, $U$ is the subgroup of $\hU$ where $\alpha_1 = 1$.

Now we consider the subgroup $\tilde{U}$ of $\SL(n)$ which is the intersection of $\SL(n)$ with the product $\hat{U} Z(\GL(n))$ of $\hat{H}$ with the central one-parameter subgroup $Z(\GL(n)) \cong \CC^*$ of $\GL(n)$. Like $\hU$, the subgroup $\tilde{U}$ of $\GL(n)$ is a semi-direct product
$$\tilde{U} = U \rtimes \CC^*$$
where $\CC^*$ acts on the Lie algebra of $U$ with all weights strictly positive.

Let $\tilde{U}=U \rtimes \CC^* \subseteq \SL(n)$ act linearly on a projective variety $X$ with respect to an ample line bundle $L$ on $X$ and assume that the action extends to a linear action of $\SL(n)$. 
 Let $\chi: \tilde{U} \to \CC^*$ be a character of $\tilde{U}$ with kernel containing $U$; we will identify $\chi $ with the integer $r_{\chi}$ such that
$$ \chi \left(\begin{array}{ccccc}t^{n\weight_1 - (\weight_1 + \cdots + \weight_n)} & 0 &0 & \ldots & 0 \\ 0 & t^{n\weight_2 - (\weight_1 + \cdots + \weight_n)}& 0 & \ldots & 0\\ 0 & 0 & t^{n\weight_3 - (\weight_1 + \cdots + \weight_n)}& \ldots & 0 \\ \cdot & \cdot & \cdot & \cdot &\cdot \\ 0 & 0 & 0 & 0 & t^{n\weight_n - (\weight_1 + \cdots + \weight_n)} \end{array}\right) = t^{r_{\chi}}.$$
Assume that the maximum
\[\max \{\weight_1+\ldots +\weight_n- \weight_{i+1}+\weight_i: \weight_{i-1}<\weight_i,1\le i \le n-1 \}\] 
is taken at $i=i_0$. Let $c$ be a positive integer such that 
$$\weight_1+\ldots +\weight_n- \weight_{i_0+1}+\weight_{i_0}-n<\frac{\chi}{c(\weight_1 + \cdots + \weight_n)}<\weight_1 + \cdots + \weight_n - n;$$ we call rational characters $\chi/c$  with this property {\it well-adapted} to the linear action.
 The linearisation of the action of $\tilde{U}$ on $X$ with respect to the ample line bundle $L^{\otimes c}$ can be twisted by the character $\chi$; 
let $L_\chi^{\otimes c}$ denote this twisted linearisation. 

The main theorem of this paper is 

\begin{theorem} \label{maina} 
Let $\hU = U \rtimes \CC^*$ be a subgroup of $\GL(n)$ which is generated along its first row with positive weights $1=\weight_1 < \weight_2 \leq \ldots \leq \weight_n$ as 
at (\ref{label1}) above, and let  $\tilde{U} = U \rtimes \CC^*$ be the intersection of $\SL(n)$ with the product $\hat{U} Z(\GL(n))$.
 Suppose that $\tilde{U}$ acts linearly on a projective variety $X$ with respect to an ample line bundle $L$ and the action extends to a linear action of $\SL(n)$. Then the algebra of invariants $\oplus_{m=0}^\infty H^0(X,L_{\chi}^{\otimes cm})^{\tilde{U}}$ is 
finitely generated for any well-adapted rational character $\chi/c$ of $\tilde{U}$.   

In addition the projective variety $X/\!/ \tilde{U}$ associated to this algebra of invariants is a categorical quotient of an open subset $X^{ss,\tilde{U}}$ of $X$ by $\tilde{U}$ and contains as an open subset a geometric quotient of an open subset $X^{s,\tilde{U}}$ of $X$.  
\end{theorem}

Applying a similar argument after replacing $X$ with $X \times \PP^1$ we can obtain geometric information on the action of the unipotent group $U$ on $X$: 

\begin{theorem} \label{cor:invariants}
In the situation above let $\tilde{U}$ act diagonally on $X \times \PP^1$ 
and linearise this action using the tensor product of $L_\chi$ with $\calo_{\PP^1}(M)$ for suitable $M \geq 1$. Then $(X \times \PP^1)/\!/\tilde{U}$ is a projective variety which is a categorical quotient by $\tilde{U}$ of a $\tilde{U}$-invariant open subset of $X \times \CC$ and contains as an open subset a geometric quotient of a $U$-invariant open subset $X^{\hat{s},U}$ of $X$ by $U$.
\end{theorem}

\begin{rem}\label{cor:invariants2}
This theorem's proof also shows that the algebra  $\oplus_{m=0}^\infty H^0(X\times \PP^1,L_{\chi}^{\otimes cm} \otimes \mathcal{O}_{\PP^1}(M))^{\tilde{U}}$ of $\tilde{U}$-invariants is 
finitely generated for a well-adapted rational character $\chi/c$ of $\tilde{U}$ when $c$ is a sufficiently divisible positive integer. This graded algebra can be identified with the subalgebra of the algebra of $U$-invariants $\oplus_{m=0}^\infty H^0(X,L^{\otimes cm})^{{U}}$  generated by those 
weight vectors  for the action of $\CC^* \leq \tilde{U}$ on
 $\oplus_{m=0}^\infty H^0(X,L^{\otimes cm})^{{U}}$ 
 with non-positive weights after twisting by the well-adapted character $\chi$.
\end{rem}

Note that if $U$ is {\it any} unipotent complex linear algebraic group which has
an action of $\CC^*$ with all weights strictly positive (that is, $U$ is graded unipotent), then $U$ can be embedded in $\GL({\rm Lie}(U \rtimes \CC^*))$ via its adjoint action on the Lie algebra
${\rm Lie}({U}\rtimes \CC^*)$ as the unipotent radical of a subgroup $\hU$ of this form 
 which is generated along the first row, and as the unipotent radical of the associated subgroup $\tilde{U}$ of $\SL(n)$. We will call this the adjoint form of $U$.  
However there are many examples (including the reparametrisation groups for jet 
differentials) of subgroups of $\GL(n)$ of the form (\ref{label1}) where the action of $U$ is
not equivalent to its adjoint action on ${\rm Lie}\hat{U}$. Theorem \ref{maina} gives us the following result for the adjoint form of a graded unipotent group:

\begin{corollary} 
Let $U$ be any unipotent complex linear algebraic group with an action of $\CC^*$ with strictly positive weights and associated $\CC^*$ extension $\hU$ and adjoint embedding in $\GL({\rm Lie}(\hU))$ . If $U$ acts linearly on a projective variety $X$ and the action extends to a linear action of $\GL(\mathrm{Lie}\hU)$ then the conclusions of Theorem \ref{cor:invariants} hold.  \end{corollary}

\begin{rem} 
In the situation when $\hat{U}=\GG_n$ we cannot apply Theorems \ref{maina} and \ref{cor:invariants} directly to the action of $\GG_n$ on the fibre $J_n(Y)_x$ of the jet bundle over $x\in Y$, where $Y$ is a $k$-dimensional complex manifold, as this fibre is not projective. However, $J_n(Y)_x$ can be identified with the set of $n\times k$ matrices with nonzero first column which forms an open subset of the affine space of all $n\times k$ matrices, and we can apply the argument to the associated projective space $\PP^{nk-1}$, on which $\GG_n$ acts linearly with respect to the hyperplane line bundle $L$. When $n \geq 2$ the fibre at $x$ of the bundle $E_n(Y)$ of jet differentials can then be identified with the algebra $\oplus_{m=0}^\infty H^0(\PP^{nk-1},L^{\otimes m})$ and the Demailly algebra $E_n(Y)_x^{\UU_n}$ of $\UU_n$-invariant jet differentials can be identified with its subalgebra $\oplus_{m=0}^\infty H^0(\PP^{nk-1},L^{\otimes m})^{{\UU_n}}$.
%
\label{Demconj} 
\end{rem}

\noindent 

We will use a generalisation of a criterion in \cite{DK} to prove Theorem \ref{maina} 
 as follows.
In  \S\ref{sec:construction} we obtain an explicit $\GL(n)$-equivariant embedding of the quasi-affine variety $\GL(n)/\hat{U}$ (which can also be identified with the quotient
of $\SL(n)/U$ by a finite central subgroup of $\hat{U} \cap \tilde{U}$) 
in the Grassmannian $\grass_n(\symk)$ of $n$-dimensional linear subspaces of
\[\symk=\CC^n \oplus \sym^{\weight_2}(\CC^n) \oplus \ldots \oplus \sym^{\weight_n}(\CC^n)\]
where $\sym^{k}(\CC^n)$ is the $k$th symmetric product of $\CC^n$.
Using Pl\"{u}cker
coordinates we thus obtain an explicit projective embedding $\GL(n)/\hat{U} \hookrightarrow \PP(\wedge^n\symk).$ In fact  $\GL(n)/\hat{U}$ embeds into the open affine subset of $ \PP(\wedge^n\symk)$ where the coordinate corresponding to the one-dimensional summand $\wedge^n\CC^n$ of $\wedge^n\symk$ does not vanish.


The advantage of this embedding lies in the fact that we can control the boundary of the orbit $\GL(n)/\hU$ in $\PP(\wsymk n)$; we will prove that 
\[\overline{(\GL(n)/\hU)} \setminus (\GL(n)/\hU) 
\]
is contained in the union of two subspaces $\PP(\mathcal{W}_{v_1})$ and $\PP(\mathcal{W}_{\det})$ of $\PP(\wedge^n\symk).$
In order to prove  Theorem \ref{maina} we show that if $X$ is
 a nonsingular complex projective variety on which $\SL(n)$ acts linearly with respect to a very ample line bundle $L$, 
and 
if the linear action of $\tilde{U} \leq \SL(n)$ on $X$ is twisted by a well-adapted rational character $\chi/c$, then $\PP(\mathcal{W}_{v_1})\times X$ and $\PP(\mathcal{W}_{\det})\times X$ are unstable with respect to the induced linear action of $\SL(n) \times \CC^*$ on $\PP(\wedge^n\symk) \times X$ for an appropriate linearisation. A similar argument applies when $X$ is replaced with $X \times \PP^1$ and enables us to prove Theorem \ref{cor:invariants}.

The layout of this paper is as follows. $\S$2 provides a very brief review of
classical geometric invariant theory and some non-reductive GIT.
 In $\S$3 we describe our embedding of $\GL(n)/\hat{U}$ into $\grass_n(\symk)$.
In $\S$4 we recall the original motivation for this work from global singularity theory and jet differentials and discuss the link between jet differentials and the curvilinear component of Hilbert schemes of points. 
In   
$\S$5  we complete the proofs of Theorems \ref{maina} and \ref{cor:invariants}. Finally in $\S$6 we consider our results in the cases of jet differentials and the adjoint forms of graded unipotent groups.
\medskip 

\noindent\textbf{Acknowledgments}
The authors thank Brent Doran, Thomas Hawes and and Rich\'ard Rim\'anyi for helpful discussions on this topic.

\section{Classical and non-reductive geometric invariant theory}

Let $X$ be a complex quasi-projective variety on which a complex
reductive group $G$ acts linearly; that is, there is a
line bundle $L$ on $X$ (which we will assume to be ample) and a lift $\mathcal{L}$ of the action of $G$ to $L$.
Then $y \in X$ is said to be {\em
semistable} for this linear action if there exists some $m > 0$
and $f \in H^0(X, L^{\otimes m})^G$ not vanishing at $y$ such that
the open subset
$$ X_f = \{ x \in X \ | \ f(x) \neq 0 \}$$
is affine ($X_f$ is automatically affine if $X$ is projective or affine), and $y$ is {\em stable} if also $f$ can be chosen so that
the action of $G$
on $X_f$ is closed with all stabilisers finite.
The open subset $X^{ss}$ of $X$ consisting of semistable points has
a quasi-projective categorical quotient $X^{ss} \to \xg$, which restricts to a geometric
quotient $X^s \to X^s/G$ of the open subset $X^s$ of stable points (see \cite{GIT} Theorem 1.10).
When
$X$ is projective then $X_f$ is affine for any
nonzero $f \in H^0(X, L^{\otimes
m})^G$ (since we are assuming $L$ to be ample), and there is an induced action of $G$ on the
homogeneous coordinate ring
\[\hat{\calo}_L(X) = \bigoplus_{m \geq 0} H^0(X, L^{\otimes m}) \]
of $X$.
The
subring ${\hat{\calo}}_L(X)^G$ consisting of the elements of ${\hat{\calo}}_L(X)$
left invariant by $G$ is a finitely generated graded complex algebra
because $G$ is reductive, and
 the GIT quotient $X/\!/G$ is the associated projective variety  $\Proj({\hat{\calo}}_L(X)^G)$ \cite{Dolg, GIT, New}. When $X$ is affine
and the linearisation of the action of $G$ is trivial then
the algebra $\mathcal{O}(X)^G$ of $G$-invariant regular functions on $X$ is
finitely generated and $X^{ss} = X$ and $X/\!/G = \Spec(\mathcal{O}(X)^G)$
is the affine variety associated to $\mathcal{O}(X)^G$.

Suppose now that $H$ is any linear algebraic
group, with unipotent radical $U \unlhd H$ (so that
$H/U$ is reductive), acting linearly on a complex projective variety $X$ with respect to an ample line bundle $L$. Then the scheme  $\Proj({\hat{\calo}}_L(X)^H)$ is not in general 
 a projective variety,
since the graded complex algebra of invariants
$${\hat{\calo}}_L(X)^H = \bigoplus_{m \geq 0} H^0(X, L^{\otimes m})^H$$
is not necessarily finitely generated, and geometric invariant theory (GIT)
cannot be extended immediately to this 
situation (cf. \cite{DK,F2,F1,GP1,GP2,KPEN,W}). However in some cases it is known that $\hat{\calo}_L(X)^U$ is finitely generated, which implies that 
\[{\hat{\calo}}_L(X)^H = \left( {\hat{\calo}}_L(X)^U \right)^{H/U}\] 
is finitely generated since $H/U$ is reductive, and then the {\em enveloping quotient} in the sense of \cite{BDHK,DK} is given by
\[X/\!/H=\Proj(\hat{\calo}_L(X)^H).\] 
Moreover, there is a morphism 
\[q: X^{ss} \to X/\!/H,\] 
where $X^{ss}$ is defined as in the reductive case 
above, which restricts to a geometric quotient 
\[q:X^s \to X^s/H\]
for an open subset $X^s \subset X^{ss}$. 
In such cases we have a GIT-like quotient $X/\!/H$ and we would like to understand it geometrically. However there is a crucial difference here from the case of reductive group actions, even though we are assuming that the invariants are finitely generated: the morphism $X^{ss} \to X/\!/H$ is not in general surjective, so we cannot describe $X/\!/H$ geometrically as $X^{ss}$ modulo some equivalence relation. 

In this paper we will study the situation when $U$ is a unipotent group with a one-parameter group of automorphisms
$\lambda:\CC^* \to \mbox{Aut}(U)$ such that the weights of
the induced $\CC^*$ action on the Lie algebra $\lieu$ of
$U$ are all strictly positive.  We will call such a group $U$ a graded unipotent group. In this situation we can form the semidirect
product
$$\hat{U} = \CC^* \ltimes U$$
given by $\CC^* \times U$ with group multiplication
$$(z_1,u_1).(z_2,u_2) = (z_1 z_2, (\lambda(z_2^{-1})(u_1))u_2).$$
Note that the centre
of $\hat{U}$ is finite and meets $U$ in the trivial subgroup, so we have an inclusion
given by the composition
$$ U \hookrightarrow \hat{U} \to \mbox{Aut}(\hat{U})
\hookrightarrow \GL(\mbox{Lie}(\hat{U})) =\GL(\CC \oplus \lieu)$$
where $\hat{U}$ maps to its group of inner automorphisms and $\lieu = \mbox{Lie}(U)$.
Thus we find that $U$ is isomorphic to a closed subgroup of the reductive
group $G=\SL(\CC \oplus \lieu)$ of the form
$$
U=\left\{\left(\begin{array}{ccccc} 1 & \a_2 & \a_3 & \ldots & \a_n \\ 0 & 1 & p_{2,3}(\a_2 \ldots ,\a_n) & \ldots & p_{2,n}(\a_2,\ldots, \a_n) \\ 0 & 0 & 1 & \ldots & p_{3,n}(\a_2,\ldots, \b_n) \\ \cdot & \cdot & \cdot & \cdot &\cdot \\ 0 & 0 & 0 & 0 & 1  \end{array}\right) 
:\a_2, \ldots, \a_n \in \CC\right\}
$$
where $n = 1 + \dim U$ and $p_{i,j}(\alpha_2,\ldots,\alpha_n)$ is a polynomial in $\alpha_2,\ldots,\alpha_n$
with complex coefficients for $1<i<j \leq n$. Our aim is to study
linear actions of subgroups $U$ and $\hat{U}$ of $\GL(n)$ of this form
(but with the embedding in $\GL(n)$ not necesssarily induced by the adjoint
action on the Lie algebra of $\hU$)
which extend to linear actions of $\GL(n)$ itself, by finding  explicit affine and projective 
embeddings of the quasi-affine varieties $\GL(n)/\hat{U}$.

In cases where the action of $\hat{U}$ on $X$ extends to an action of $\GL(n)$ there is an isomorphism of
$\GL(n)$-varieties 
\begin{equation} \label{9Febiso} G \times_U X \cong (G/U) \times X
\end{equation} given by
$ [g,x] \mapsto (gU, gx)$. Here $\GL(n) \times_{\hU} X$ denotes the quotient of $\GL(n) \times X$
by the free action of $\hU$ defined by $u(g,x)=(g u^{-1}, ux)$ for $u \in \hU$,
which is a quasi-projective variety by \cite{PopVin} Theorem 4.19. Then there
is an induced $\GL(n)$-action on $\GL(n) \times_{\hU} X$ given by left
multiplication of $\GL(n)$ on itself.

Geometric invariant theory for linear actions of a unipotent group $U$ on a projective variety was studied in \cite{DK}. If $U$ is a unipotent subgroup of the reductive group $G$ then $U$-invariants on $X$ can be related to $G$-invariants of appropriate projective compactifications $\overline{G \times_U X}$ of the quasi-projective variety $G\times_U X$ where $\overline{G \times_U X}$ has a suitable $G$-linearisation extending the linearisation for the $U$ action on $X$. 


\begin{theorem}{(\cite{DK} Corollary 5.3.19)}\label{thm:fgcriterion}
Let $X$ be a nonsingular complex projective variety on which $U$ acts linearly with respect to an ample line bundle $L$. Let $L'$ be a $G$-linearisation over a projective completion $\overline{G \times_{U} X}$ of $G \times_{U} X$ extending the $G$ linearisation over $G \times_{U} X$ induced by $L$.
Let $D_1, \ldots , D_r$ be the codimension $1$ components of the boundary of $G
\times_U X$ in $\overline{G \times_U X}$ and 
suppose that for all sufficiently divisible $N$  
$L'_N=L^{\prime}[N
\sum_{j=1}^r D_j]$ is an ample line bundle on $\overline{G \times_{U} X}$. 
Then the algebra of invariants 
$\bigoplus_{k \geq 0} H^0( X,L^{\otimes k})^U$ is finitely generated if and only if  
for all sufficiently divisible $N$ any $G$-invariant section of a positive tensor power of $L'_N$ vanishes on every codimension one component $D_j$. 

\end{theorem}

\begin{rem} 
This result appears in \cite{DK} as a corollary to a theorem (Theorem 5.3.18 in \cite{DK}) which claims without the additional hypothesis that $L'_N$ is ample for sufficiently large $N$, that $\bigoplus_{k \geq 0} H^0( X,L^{\otimes k})^U$ is finitely generated if and only if any $G$-invariant section of a positive tensor power of $L'_N$ vanishes on every codimension $1$ component $D_j$ in boundary of $G\times_H X$ in $\overline{G\times_H X}$.

However, there is an error in the proof of that theorem: it requires the algebra of $G$-invariants $\oplus_{k\ge 0}H^0(\overline{G\times_H X},(L'_N)^{\otimes k}$ to be finitely generated . Since $G$ is reductive and $\overline{G\times_H X}$ is projective, this is true when $L'_N$ is an ample line bundle for sufficiently divisible $N$, but does not follow in general without assuming such additional hypothesis. Since \cite{DK} Corollary 5.3.19 includes the hypothesis that $L'_N$ is ample line bundle for $N$ sufficiently divisible, its validity is unaffected by this error. 
\end{rem}
Theorem \ref{thm:fgcriterion} can be generalised to allow us to study $H$-invariants for linear algebraic groups $H$ which are neither unipotent nor reductive. Over $\CC$ any linear algebraic group $H$ is a semi-direct product $H=U_H\rtimes R$ where $U_H \subset H$ is the unipotent radical of $H$ (its maximal unipotent subgroup) and $R\simeq H/U_H$ is a reductive subgroup of $H$. If $H$ is a subgroup of a reductive group $G$ then there is an induced right action of $R$ on $G/U_H$ which commutes with the left action of $G$. 
Similarly if $H$ acts on a projective variety $X$ then there is an induced action of $G\times R$ on $G\times_{U_H}X$ with an induced $G\times R$-linearisation. The same is true if we replace the requirement that $H$ is a subgroup of $G$ with the existence of a group homomorphism $H\to G$ whose restriction to $U_H$ is injective.

\begin{definition}
A group homomorphism $H \to G$ from a linear algebraic group $H$ to a reductive group $G$ will be called $U_H$-faithful if its restriction to the unipotent radical $U_H$ of $H$ is injective.  
\end{definition} 

The proof of \cite{DK} Theorem 5.1.18 gives us 

\begin{theorem}\label{thm:fgcriteriongeneral}
Let $X$ be a nonsingular complex projective variety acted on by a linear algebraic group $H=U_H\rtimes R$ where $U_H$ is the unipotent radical of $H$ and let $L$ be a very ample linearisation of the $H$ action defining an embedding $X\subseteq \PP^n$. Let $H \to G$ be an $U_H$-faithful homomorphism into a reductive subgroup $G$ of $\SL(n+1)$ with respect to an ample line bundle $L$. Let $L'$ be a $G\times R$-linearisation over a normal nonsingular projective completion $\overline{G \times_{U_H} X}$ of $G \times_{U_H} X$ extending the $G\times R$ linearisation over $G \times_{U_H} X$ induced by $L$. 
Let $D_1, \ldots , D_r$ be the codimension one components of the boundary of $G
\times_{U_H} X$ in $\overline{G \times_{U_H} X}$, and suppose for all sufficiently divisible $N$  that
$L'_N=L^{\prime}[N
\sum_{j=1}^r D_j]$ is an ample line bundle on $\overline{G \times_{U_H} X}$. 
Then the algebra of invariants 
$\bigoplus_{k \geq 0} H^0( X,L^{\otimes k})^H$ is finitely generated if and only if  
for all sufficiently divisible $N$ any $G\times R$-invariant section of a positive tensor power of $L'_N$ vanishes on every codimension one component $D_j$. 
\end{theorem}

\begin{proof}
For the forward direction first note that by restriction
$$\bigoplus_{k \geq 0} H^0( \overline{G \times_{U_H} X},(L_N')^{\otimes k})^{G \times R} \subseteq 
\bigoplus_{k \geq 0} H^0( {G \times_{U_H} X},(L_N')^{\otimes k})^{G \times R} 
= (\bigoplus_{k \geq 0} H^0( {G \times_{U_H} X},(L_N')^{\otimes k})^{G})^{  R} $$
$$\cong 
(\bigoplus_{k \geq 0} H^0( { X},L^{\otimes k})^{U_H})^ { R} = \bigoplus_{k \geq 0} H^0( { X},L^{\otimes k})^H. $$ 
We can identify $H^0( \overline{G \times_{U_H} X},L_n')$ with a subspace of $H^0( \overline{G \times_{U_H} X},L_{n+1}')$ for
any natural number $n$, so that a section $f$ of $L_{n+1}')$ extends to a section $F$ of $L_{n}')$ if and only if it vanishes on each $D_j$ as a section of $L_{n+1}')$. Any  given $G\times R$-invariant section of 
$L^{\otimes k}$
over $G \times_{U_H} X$ extends to a section of $(L_N')^{\otimes k}$ over each $D_j$ for large enough $N$ and thus by normality extends over $ \overline{G \times_{U_H} X}$ (cf. the proof of Converse 1.13 on page 41 of  \cite{GIT}).
So if the algebra of invariants $\bigoplus_{k \geq 0} H^0( { X},L^{\otimes k})^H$ is finitely generated, for large enough $N$ the finitely many generators will 
 all extend over and vanish on every $D_j$ as a section of a tensor power of $L_N'$, and  hence every element of 
$\bigoplus_{k \geq 0} H^0( { X},L^{\otimes k})^H$ will have the same property.

The reverse direction follows by proving that for any such $N$ the ring of invariants
$$
\bigoplus_{k \geq 0} H^0( { X},L^{\otimes k})^H \cong
\bigoplus_{k \geq 0} H^0( {G \times_{U_H} X},(L_N')^{\otimes k})^{G \times R}
 $$
 is isomorphic to the ring of invariants 
$\bigoplus_{k \geq 0} H^0( \overline{G \times_{U_H} X},(L_N')^{\otimes k})^{G \times R}$, which is finitely generated since $ \overline{G \times_{U_H} X}$ is
a projective variety acted on linearly 
with respect to the ample linearisation $L′_N$ by the reductive group $G$. This isomorphism
￼￼￼￼arises since any  $G \times R$-invariant section $s$ over $G \times_{U_H} X$ of $L′_N$ extends as above to an invariant section of $L_{N′}'$ over $\overline{G \times_{U_H} X}$ for some $N′ > N$. By hypothesis this section vanishes on each $D_j$ and hence defines an invariant section of $L_{N'-1}'$ extending $s$. Repeating this argument enough times we find that
$s$ extends to a section of $L′_N$. The same argument applies to any invariant section $s$ over $G \times_{U_H} X$ of a positive tensor power $(L′_N)^{\otimes k}$ of $L′_N$, so we have 
$$\bigoplus_{k \geq 0} H^0( { X},L^{\otimes k})^H \cong
\bigoplus_{k \geq 0} H^0( \overline{G \times_{U_H} X},(L_N')^{\otimes k})^{G \times R}$$ as required.
\end{proof}

\begin{rem}
The proof of Theorem \ref{thm:fgcriteriongeneral} tells us that when the hypotheses hold and the algebra of invariants $\bigoplus_{k \geq 0} H^0( X,L^{\otimes k})^H$ is finitely generated then the enveloping quotient
\begin{equation}\label{equotient}
X/\!/H =\mathrm{Proj}(\oplus_{k \geq 0} H^0( X,L^{\otimes k})^H)\simeq \overline{G \times_{U_H} X}/\!/_{L'_N} (G\times R)
\end{equation}
for sufficiently divisible $N$. 
\end{rem}

\begin{rem}
Note that in this argument there is in fact no requirement for $U=U_H$ to be the full unipotent radical of $H$; all we need  is that $U$ is a normal subgroup of  
$H$ and $R=H/U$ is reductive.  
\end{rem}

In general even when the algebra of invariants $\bigoplus_{k \geq 0} H^0( X,L^{\otimes k})^H$ on $X$ is finitely generated and \eqref{equotient} is true, the morphism $X \to X/\!/_eH$ is not surjective and in order to study the geometry of $X/\!/_eH$ by identifying it with $\overline{G \times_{U_H} X}/\!/_{L'_N} (G\times R)$ we need information about the boundary $\overline{G \times_{U_H} X}\setminus G\times_{U_H} X$ of $\overline{G \times_{U_H} X}$. If, however, we are lucky enough to to find a $G \times R$-equivariant projective completion $\overline{G \times_{U_H} X}$ with a linearisation $L$ such that for sufficiently large $N$ $L'_N$ is an ample line bundle and the boundary $\overline{G \times_{U_H} X}\setminus G\times_{U_H} X$ is unstable for $L'_N$ then we have a situation which is almost as well behaved as for reductive group actions on projective varieties with ample linearisations as follows.

\begin{definition} Let $X^{\barss}=X\cap \overline{G \times_{U_H} X}^{ss,G\times R}$ and $X^{\bars}=X\cap \overline{G \times_{U_H} X}^{s,G\times R}$
where $X$ is embedded in $G \times_{U_H} X$ in the obvious way as $x\mapsto [1,x]$. 
\end{definition}

\begin{theorem}\label{thm:geomcor} 
Let $X$ be a complex projective variety acted on by a linear algebraic group $H=U_H\rtimes R$ where $U_H$ is the unipotent radical of $H$ and let $L$ be a very ample linearisation of the $H$ action defining an embedding $X\subseteq \PP^n$. Let $H \to G$ be an $U_H$-faithful homomorphism into a reductive subgroup $G$ of $\SL(n+1)$ with respect to an ample line bundle $L$. Let $L'$ be a $G\times R$-linearisation over a projective completion $\overline{G \times_{U_H} X}$ of $G \times_{U_H} X$ extending the $G\times R$ linearisation over $G \times_{U_H} X$ induced by $L$. 
Let $D_1, \ldots , D_r$ be the codimension $1$ components of the boundary of $G
\times_{U_H} X$ in $\overline{G \times_{U_H} X}$, and suppose that some
integral multiple of each $D_j$ is Cartier and for all sufficiently divisible $N$  that 
$L'_N=L^{\prime}[N
\sum_{j=1}^r D_j]$ is an ample line bundle on $\overline{G \times_{U_H} X}$. If 
for all sufficiently divisible $N$ any $G\times R$-invariant section of a positive tensor power of $L'_N$ vanishes on the 
 boundary of $G
\times_{U_H} X$ in $\overline{G \times_{U_H} X}$, then
\begin{enumerate}
\item the algebra of invariants 
$\bigoplus_{k \geq 0} H^0( X,L^{\otimes k})^H$ is finitely generated;
\item the enveloping quotient $X/\!/H \simeq \overline{G \times_{U_H} X}/\!/_{L'_N} (G\times R)\simeq \mathrm{Proj}(\oplus_{k \geq 0} H^0( X,L^{\otimes k})^H)$ for sufficiently divisible $N$;
\item $\overline{G \times_{U_H} X}^{ss,G\times R, L'_N} \subseteq G\times_{U_H} X$ and therefore the morphism 
\[ \phi: 
X^{\barss} \rightarrow X/\!/H\]
is surjective and $X/\!/H$ is a categorical quotient of $
X^{\barss}$;
\item if $x,y \in 
X^{\barss}$ then $\phi(x) = \phi(y)$ if and only if the closures of the $H$-orbits of $x$ and $y$ meet in $
X^{\barss}$;
\item $\phi$ restricts to a geometric quotient $X^{\bars} \rightarrow X^{\bars}/H \subseteq  X/\!/H$.
\end{enumerate} 
\end{theorem}

\begin{rem}
Note that the hypotheses in Theorem \ref{thm:fgcriteriongeneral} that $X$ should be nonsingular and that $\overline{G \times_{U_H} X}$ should be normal and nonsingular are not needed in Theorem 
\ref{thm:geomcor}. This is because these hypotheses are only 
required to ensure that sections extend over irreducible components of codimension at least two in the boundary which are not unstable; in the circumstances of Theorem \ref{thm:geomcor} there are no such irreducible components.
\end{rem}

\begin{proof} 
If $N$ is sufficiently divisible then the composition 
\begin{equation}\label{comp}
X^{\barss}\to \overline{G \times_{U_H} X}^{ss,G\times R,L'_N} \to \overline{G \times_{U_H} X}/\!/_{L'_N}(G\times R)\end{equation}
is an $H$-invariant morphism, and $\overline{G \times_{U_H} X}/\!/_{L'_N}(G\times R)$ has an ample line bundle $\mathcal{L}$ which pulls back to a positive tensor power $L^{\otimes r}$ of the restriction to $X^{\barss}$ of the linearisation $L$ of the $H$ action on $X$.  

$X^{\barss}$ is an open subset of $X^{ss,fg}$ and \eqref{comp} factors through the quotient map 
\[q:X^{ss,fg} \to q(X^{ss,fg})\subseteq \mathcal{U}\]
where $\mathcal{U}$ is a quasi-projective open subset of the enveloping quotient $X/\!/_eH$ with a birational morphism $\tau: \mathcal{U} \dasharrow  \overline{G \times_{U_H} X}/\!/_{L'_N}(G\times R)$
as
\[X^{\barss} \hookrightarrow X^{ss,fg} \xrightarrow{q} \mathcal{U} \xrightarrow{\tau} \overline{G \times_{U_H} X}/\!/_{L'_N}(G\times R).\]
The hypothesis that the boundary of $\overline{G \times_{U_H} X}$ is unstable ensures that this composition is surjective. Moreover, $\overline{G \times_{U_H} X}/\!/_{L'_N}(G\times R)$ is a categorical quotient of $G\times_{U_H} X^{\barss}$ by $G\times R=G\times (H/U_H)$, so it is also a categorical quotient of $G\times X^{\barss}$ by $G\times H$ and a categorical quotient of $X^{\barss}$ by $H$, and (4) and (5) now follow from the analogous properties for classical GIT applied to the reductive group $G \times R$. The $H$-invariant morphism 
$q: X^{ss,fg} \to \mathcal{U}$ then factors through a birational morphism 
\[\sigma: \overline{G \times_{U_H} X}/\!/_{L'_N}(G\times R) \to \mathcal{U}.\]
Since $\s$ is surjective and $\overline{G \times_{U_H} X}/\!/_{L'_N}(G\times R)$ is projective it follows that $\mathcal{U}$ is projective which means that $\mathcal{U}=X/\!/_e H$ and $q$ is surjective. Furthermore $\s$ and $\tau$ are mutually inverse isomorphisms between $X/\!/_e H$ and $\overline{G \times_{U_H} X}/\!/_{L'_N}(G\times R)$. Finally, since 
\[\overline{G \times_{U_H} X}^{ss,G\times R, L'_N} \subset G\times_{U_H}X\] 
we have 
\begin{multline}
\bigoplus_{k\ge 0} H^0(X,L^{\otimes rk})^H \simeq \bigoplus_{k\ge 0} H^0(G \times_{U_H} X,(L'_N)^{\otimes rk})^{G\times R}\simeq \\
\bigoplus_{k\ge 0} H^0(\overline{G \times_{U_H} X}^{ss,G\times R,L'_N},(L'_N)^{\otimes rk})^{G\times R}\simeq 
\oplus_{k\ge 0} H^0(\overline{G \times_{U_H} X}/\!/_{L'_N}(G\times R),\mathcal{L}^{\otimes k}).
\end{multline}
Thus $\bigoplus_{k\ge 0} H^0(X,L^{\otimes rk})^H$ is a finitely generated graded algebra and 
\[X/\!/_eH \simeq \overline{G \times_{U_H} X}/\!/_{L'_N} (G\times R)\simeq \mathrm{Proj}(\oplus_{k \geq 0} H^0( X,L^{\otimes k})^H).\]
\end{proof}

\begin{rem}
Note that in the circumstances of Theorem \ref{thm:geomcor} so that $\overline{G \times_{U_H} X}^{ss,G\times R, L'_N}=G\times_{U_H} 
X^{\barss}$ we get a nice geometric description of $X/\!/H$. We know from classical GIT that  the morphism from 
$\overline{G \times_{U_H} X}^{ss,G\times R, L'_N}=G\times_{U_H} 
X^{\barss}$ to $X/\!/H$ is $G$-invariant and surjective, and maps two points of $X^{\barss}$ to the same point of $X/\!/H$ if and only if the closures of their $G \times R$-orbits meet in $\overline{G \times_{U_H} X}^{ss,G\times R, L'_N}=G\times_{U_H} X^{\barss}$. Since the $G \times R$-sweep in $G\times_{U_H} X^{\barss}$ of any closed $H$-invariant subset of $X^{\barss}$ is closed in $G\times_{U_H} X^{\barss}$, it follows that the $H$-invariant morphism
$\phi: X^{\barss} \twoheadrightarrow X/\!/_LH$ is surjective and if $x_1,x_2\in X^{\barss}$ then $\phi(x_1)=\phi(x_2)$ if and only if $\overline{Hx_1} \cap \overline{Hx_2}\cap X^{\barss}\neq \emptyset$, as in Theorem \ref{thm:geomcor} (3) and (4). We can also use the Hilbert--Mumford criteria for (semi)stability from classical GIT to determine the subsets $X^{\bars}$ and $X^{\barss}$ of $X$ in an analogous way.
\end{rem}

\subsection{Symplectic geometry of $X/\!/H$}

Suppose that the action of $H$ on $X$ extends to a linear action of $G$ on $X$ and that the projective completion $\overline{G \times_{U_H} X}$ is of the form $\overline{G/U_H} \times X$ where $\overline{G/U_H}$ is a $G\times R$-equivariant projective completion of $G/U_H$ and $G \times_{U_H} X$ is identified with $G/U_H \times X$ via the $G\times R$-equivariant isomorphism
\[[g,x] \mapsto (gU_H, gx).\]
If furthermore $K$ is a maximal compact subgroup of $G$ such that $K_R=K\cap R$ is a maximal compact subgroup of $R$ then we can give a moment map description of $X/\!/H$. For this we choose coordinates for the projective embedding of $X$ defined by $L^N$ and of $\overline{G/U_H}$ such that $K$ acts unitarily. Then we have moment maps
\[\mu_X:X \to k^* \text{ and } \mu_{\overline{G/U_H}}: \overline{G/U_H} \to k^* \times k_R^*\] 
for the actions of $K$ on $X$ and of $K\times K_R$ on $\overline{G/U_H}$ such that the moment map for the action of $K \times K_R$ on $\overline{G/U_H} \times X$ with respect to $L'_N$ is given by
\[\mu:(y,x) \mapsto (N\mu_{\overline{G/U_H}}(y)+\mu_X(x),0) \in k^* \times k^*_R\]
We can identify $X/\!/H$ with 
\[\mu^{-1}(0)/(K \times K_R)=\left\{(y,x)\in (\pi_{k_R} \circ \mu_{\overline{G/U_H}})^{-1}(0) \times X:\mu_X(x)=-N\pi_k\mu_{\overline{G/U_H}}(y)\right\}/(K \times K_R)\]
where 
\[\pi_k: k^* \times k^*_R \to k^* \text{ and } \pi_{k_R}: k^* \times k^*_R \to k^*_R \]
are the projections.
Given a good understanding of the moment map $\mu_{\overline{G/U_H}}:\overline{G/U_H} \to k^*\times k_R^*$ this can provide a nice description of $X/\!/H$ in terms of $\mu_X$.

\begin{ex}
When the unipotent radical $U_H$ of $H$ is a maximal unipotent subgroup of $G$ we can use 
 the theory of symplectic implosion, due to Guillemin,
Jeffrey and Sjamaar \cite{Guillemin-JS:implosion} (or more generally when $U$ is the unipotent radical of a parabolic subgroup of $G$ we can use a generalised version of symplectic implosion \cite{Ksympimpl}).  

Let us choose a
$K$-invariant inner product on the Lie algebra \( \kf \) of a maximal compact subgroup \( K \) of $G$, which
allows us to identify \( \kf \) with its dual \( \kf^* \). Let $\tf_+$ be
a positive Weyl chamber in the Lie algebra \( \tf \) of a maximal torus \( T \) of $K$.
Given a symplectic
manifold \( M \) with a Hamiltonian symplectic action of  \( K \), the implosion \(
M_{\mathrm{impl}} \) is a stratified symplectic space with a Hamiltonian action
of the maximal torus \( T \) of \( K \), such that there is an identification of reduced spaces
$$ 
  M \symp_\lambda^s K   = M_\impl \symp^s_\lambda T  = (M
\times \mathsf O_{-\lambda}) \symp_0^s K = \mu^{-1}(\lambda)/{\Stab_K(\lambda)}
$$ 
for all \( \lambda \) in the closure of the fixed positive Weyl
chamber in \( \tf^* \), where \( \symp_\lambda^s \)
 denotes symplectic
reduction at level \( \lambda \)  
and  \( \mathsf
O_\lambda \) is the coadjoint orbit of \( K \) through \( \lambda \)
with its canonical symplectic structure, while   
 \(
\mu\colon M \to \kf^* \) is the moment map for the \( K \)-action on
\( M \) and \( \Stab_K(\lambda) \) is the stabiliser in \( K \) of \(
\lambda \in \kf^* \) under the coadjoint action of \( K \).

When \( M \)
is the cotangent bundle \( T^*K \) (which may be identified with \(
G = K_\CC \)) then  \( (T^*K)_\impl \) is obtained from \( K \times \tf_{+}
\) by identifying \( (k_1, \xi) \) with \( (k_2, \xi) \) if \( k_1,
k_2 \) lie in the same orbit of the
commutator subgroup of \( \Stab_K(\xi)\). If \( \xi \) is in the
interior of the chamber, its stabiliser is a torus and no
 identifications are made: an open dense subset of \(
(T^*K)_\impl \) is just the product of \( K \) with the interior of
the Weyl chamber.

As  \( T^*K
\) has a Hamiltonian \( K \times K \)-action 
its implosion inherits a
Hamiltonian \( K \times T \)-action. The moment map for the $K$-action is induced by the map
 \( K \times \tf_{+} \to \kf \cong \kf^*
\)  given by $(k,\xi) \mapsto k(\xi)$ while the moment map for the $T$-action is induced by the projection onto $\tf_+ \subseteq \tf \cong \tf$.
  For a general symplectic
manifold \( M \) with a Hamiltonian \( K \)-action  the
imploded space \( M_\impl \) is the symplectic quotient \( (M \times
(T^*K)_\impl) \symp_0^s K \), with its induced Hamiltonian \( T
\)-action. This can be obtained from $\mu^{-1}(\tf_+)$ by identifying $x$ with $y$ if $\mu(x) = \mu(y) = \xi$ and furthermore $x$ and $y$ lie in the same orbit of the
commutator subgroup of \( \Stab_K(\xi)\).

 \( (T^*K)_\impl \) can be identified with the affine
variety which is the non-reductive GIT quotient
\begin{equation*}
  K_\CC \symp U = \Spec(\calo(K_\CC)^U) = \overline{G/U},
\end{equation*}
 of the complex
reductive group \( G = K_\CC \) 
by a maximal unipotent subgroup \( U \); here $\calo(K_\CC)^U$ is always finitely generated. This variety
has a stratification by quotients of \( K_\CC \) by commutators of
parabolic subgroups; the open stratum is just \( K_\CC/U \) and \( K_\CC
\symp U \) is the canonical affine completion of the quasi-affine
variety \( K_\CC / U \). Thus when $G$ acts linearly on a projective variety $X$ with an ample linearisation $L$, then the enveloping quotient $X/\!/ U$ has a description in terms of the corresponding moment map $\mu_{X,K}: X \to \kf^*$: it can be obtained from $\mu_{X,K}^{-1}(\tf_+)$ by identifying $x$ with $y$ if $\mu_{X,K}(x) = \mu_{X,K}(y) = \xi$ and furthermore $x$ and $y$ lie in the same orbit of the
commutator subgroup of \( \Stab_K(\xi)\). There is a similar description for $X/\!/U$ when $U$ is the unipotent radical of any parabolic subgroup of $G$. Moreover when $H$ is a subgroup of the normaliser of $U$ in $G$ with reductive quotient $R = H/U$ which can be identified with the complexification of a subgroup $K_R$ of $K$, then we get an induced moment map for the action of $K \times K_R$ on $X \times \overline{G/U} $ and thus a description of $X/\!/H$ in terms of $\mu_{X,K}$ and the moment map $\mu_R$ for the action of $K \cap R$ on  $\overline{G/U}$. In the situation of Theorem \ref{thm:geomcor} we can identify $X/\!/H$ with $\mu_{X,K}^{-1}(\tf_+^R)/(K \cap R)$ where $\tf_+^R$ is a $K \cap R$-invariant subset of $\tf_+$ whose intersection with the image of $\mu_{X,K}$ does not meet the boundary of $\tf_+$.
\end{ex}

\section{Embeddings in Grassmannians}\label{sec:construction}
Let $U$ be a unipotent subgroup of the complex special linear group $\SL(n)$ and 
let $\hU=U \rtimes \CC^*$
be a subgroup of the complex general linear group $\GL(n)$ which is a $\CC^*$-extension of $U$ such that the weights of the $\CC^*$ action on $\mathrm{Lie}(U)$ are all strictly positive. Let us suppose also that $U$ and $\hU$ are upper triangular subgroups of $\GL(n)$ which are generated along the first row; that is, there are integers $1 = \weight_1 < \weight_2 \leq \weight_3 \leq \cdots \leq \weight_n$
and polynomials $p_{i,j}(\alpha_1,\ldots,\alpha_n)$ in $\alpha_1,\ldots,\alpha_n$
with complex coefficients for $1<i<j \leq n$ such that 
\begin{equation}\label{presentation}
\hU=\left\{\left(\begin{array}{ccccc}\a_1 & \a_2 & \a_3 & \ldots & \a_n \\ 0 & \a_1^{\weight_2} & p_{2,3}(\ba) & \ldots & p_{2,n}(\ba) \\ 0 & 0 & \a_1^{\weight_3} & \ldots & p_{3,n}(\ba) \\ \cdot & \cdot & \cdot & \cdot &\cdot \\ 0 & 0 & 0 & 0 & \a_1^{\weight_n}  \end{array}\right)
: \ba =(\a_1,\ldots, \a_n) \in \CC^* \times \CC^{n-1} \right\}
\end{equation}
and 
$$
U=\left\{\left(\begin{array}{ccccc} 1 & \a_2 & \a_3 & \ldots & \a_n \\ 0 & 1 & p_{2,3}(\a) & \ldots & p_{2,n}(\a) \\ 0 & 0 & 1 & \ldots & p_{3,n}(\a) \\ \cdot & \cdot & \cdot & \cdot &\cdot \\ 0 & 0 & 0 & 0 & 1  \end{array}\right) 
:\a=(1,\a_2, \ldots, \a_n) \in \CC^{n-1}\right\}.
$$
This implies that the Lie algebra $\lieu = {\rm Lie}(U)$ has a similar form:
$$
\lieu=\left\{\left(\begin{array}{ccccc} 0 & a_2 & a_3 & \ldots & a_n \\ 0 & 0 & q_{2,3}(a) & \ldots & q_{2,n}(a) \\ 0 & 0 & 0 & \ldots & q_{3,n}(a) \\ \cdot & \cdot & \cdot & \cdot &\cdot \\ 0 & 0 & 0 & 0 & 0  \end{array}\right) 
:a=(a_2, \ldots, a_n) \in \CC^{n-1}\right\}
$$
where the $q_{i,j}$ are linear forms in the parameters $a=(a_2,\ldots, a_n)  \in \CC^{n-1}$ satisfying the following properties:
\begin{enumerate}
\item[(i)] $q_{i,j}=0$ for $i\le j$.
\item[(ii)] Let $\weightlie_i=\weight_i-1$ for $i=1,\ldots n$ be the weights of the adjoint
action of the subgroup $\CC^*$ of $\hat{U}$ on $\hat{\lieu}=\Lie \hU$,
so that $\weightlie_1=0$ and $\weightlie_i > 0$ if $i=2,\ldots, n$. Then the $\CC^*$-action
makes $\lieu = \Lie U$ into a graded Lie algebra, 
and therefore
\begin{equation} \label{structure}
q_{i,j}(a_2,\ldots, a_n)=\sum_{\ell:\weightlie_\ell+\weightlie_i=s_j}c_j^{\ell i}a_\ell
\end{equation} 
for some structure coefficients $c_j^{\ell i} \in \CC$. 
\end{enumerate}


\begin{rem} \label{rmk3.2} 
In particular, \eqref{structure} implies that $c^{\ell i}_j=0$ for $\ell\ge j$ unless $i=1$. But $q_{1,j}=a_j$ so this means that for $i\ge 2$ 
\[q_{i,j}(a_2,\ldots, a_n)=q_{ij}(a_2,\ldots, a_{j-1})\]
is a linear form in the first $j-1$ free parameters. It follows immediately that for  
$j \geq i\geq 2$ 
\[p_{i,j}(\ba) = p_{i,j}(\a_1, \ldots, \a_{j-1})\]
depends only on $\a_1, \ldots, \a_{j-1}$. 
\end{rem}

\begin{prop}\label{homogprop}
 Let the weighted degree of $\a_s$ be $\deg(\a_s)=\weight_s$
for $1 \leq s \leq n$. Then
\begin{enumerate} 
\item[(i)] the polynomial $p_{i,j}(\ba)$ which is the $(i,j)$th entry of the element of $\hat{U}$ parametrised by $\ba = (\a_1,\ldots,\a_n) \in
\CC^* \times \CC^{n-1}$ is homogeneous of degree $\weight_i$ in $\a_1,\ldots,\a_n$;
\item[(ii)] $p_{i,j}(\ba)$ is weighted homogeneous of degree $\weight_j$
in $\a_1,\ldots,\a_n$. 
\end{enumerate}
\end{prop}

\begin{proof}
The first (respectively second) part of the statement follows from the fact that $\hU$ is closed under multiplication by its subgroup
\[ \CC^*=\left \{ \left(\begin{array}{cccc} \a_1 & 0 & \ldots & 0 \\ 0 & \a_1^{\weight_2} & \ldots & 0 \\ \cdot & \cdot & \cdot & \cdot \\ 0 & \cdot & \cdot & \a_1^{\weight_n} \end{array} \right): \a_1 \in \CC^* \right \}\]
on the left (respectively right).
\end{proof}

\subsection{The construction}

For a vector of positive integers $\bk=(\weight_1,\ldots, \weight_n)$ we introduce the notation 
\[\symk=\CC^n \oplus \sym^{\weight_2}(\CC^n) \oplus \ldots \oplus \sym^{\weight_n}(\CC^n),\]
where $\sym^{s}(\CC^n)$ is the $s$th symmetric power of $\CC^n$. Any linear group action on $\CC^n$ induces an action on $\symk$.
 
The most straightforward way to find an algebraic description of the quotient $\GL(n)/\hU$ is to find a $\GL(n)$-module $W$ with a point $w \in W$ whose stabiliser is $\hU$. Then the orbit $\GL(n)\cdot w$ is isomorphic to $\GL(n)/\hU$ as a quasi-affine variety, and its closure $\overline{\GL(n)\cdot w}$ in $W$ is an
affine completion of $\GL(n)/\hU$, while its closure in a projective
completion of $W$ is a compactification of $\GL(n)/ \hU$.

\begin{theorem} 
\label{embed} Let $\hU=U \rtimes \CC^*$ be a $\CC^*$ extension of a unipotent subgroup $U$ of  $\SL(n)$ with positive weights $1=\weight_1 < \weight_2 \le \cdots \le \weight_n$ and  a polynomial presentation \eqref{presentation}. 
Fix the standard basis $\cale=\{e_1,\ldots, e_n\}$ of $\CC^n$ and define 
\begin{equation}\label{pn}
\zdis_n=[e_1 \wedge (e_2+e_1^{\weight_2}) \wedge \ldots \wedge (e_j + \sum_{i=2}^j p_{i,j}(e_1,\ldots,e_{j-1}))\wedge  \ldots \wedge (e_n + \sum_{i=2}^n p_{i,n}(e_1,\ldots,e_n))]
\end{equation}
$$ \in \grass_n(\symk)\subset \PP(\wedge^n\symk).$$
Then the stabiliser $\GL(n)_{\zdis_n}$ of $\zdis_n$ in $\GL(n)$ is $\hU$.
\end{theorem}

\begin{corollary}\label{phi}
The map $\phi_n:\GL(n) \to \PP[\wedge^n \symk]$ which sends a matrix with column vectors $v_1,\ldots, v_n$ to the point 
\begin{equation}\label{phidef}
(v_1,\ldots, v_n) \mapsto [v_1 \wedge (v_2+v_1^{\weight_2}) \wedge \ldots \wedge (v_n + \sum_{i=2}^n p_{i,n}(v_1,\ldots,v_n))]
\end{equation}
is invariant under right multiplication of $\hU$ on $\GL(n)$ and $\GL(n)$-equivariant with respect to left multiplication on $\GL(n)$ and the induced action on $\PP[\wedge^n\symk]$. It therefore defines a $\GL(n)$-equivariant embedding 
\begin{equation}\label{embedding}
\phi_n: \GL(n)/\hU \hookrightarrow \grass_n(\symk).
\end{equation}
\end{corollary}

\begin{rem} \label{afemb}
Note that the image of the embedding 
$\phi_n:\GL(n) \to \PP[\wedge^n \symk]$ lies in the open affine subset defined by the
non-vanishing of the
coordinate in $\wedge^n \symk$ corresponding to the one-dimensional summand
$\wedge^n \CC^n$ of $\wedge^n \symk$ spanned by $e_1 \wedge \cdots \wedge e_n$.  
\end{rem}

\begin{proof}[Proof of Theorem \ref{embed}]
First we prove that  $\hU$ is contained in the stabiliser 
$\GL(n)_{\zdis_n}$.
For $(\a_1,\ldots, \a_n)\in \CC^* \times \CC^{n-1}$ let 
\[u(\a_1, \ldots, \a_n) =\left(\begin{array}{ccccc}\a_1 & \a_2 & \a_3 & \ldots & \a_n \\ 0 & \a_1^{\weight_2} & p_{2,3}(\ba) & \ldots & p_{2,n}(\ba) \\ 0 & 0 & \a_1^{\weight_3} & \ldots & p_{3,n}(\ba) \\ \cdot & \cdot & \cdot & \cdot &\cdot \\ 0 & 0 & 0 & 0 & \a_1^{\weight_n}  \end{array}\right)\in \hU
\]
denote the element of $\hU$ determined by the parameters $(\alpha_1, \ldots, \alpha_n)$ and for an $n$-tuple of vectors $\bv=(v_1,\ldots, v_n)\in (\CC^n)^{\oplus n}$ forming the columns of the $n \times n$-matrix $A\in \GL(n)$ we similarly define the matrix
\[u(A)=u(v_1, \ldots, v_n) =\left(\begin{array}{ccccc}v_1 & v_2 & v_3 & \ldots & v_n \\ 0 & v_1^{\weight_2} & p_{2,3}(\bv) & \ldots & p_{2,n}(\bv) \\ 0 & 0 & v_1^{\weight_3} & \ldots & p_{3,n}(\bv) \\ \cdot & \cdot & \cdot & \cdot &\cdot \\ 0 & 0 & 0 & 0 & v_1^{\weight_n}  \end{array}\right)\in M_{n \times n}(\symk)
\]
with entries in $\symk$. Then the map $\phi$ in \eqref{phidef} is the composition
\[\phi(v_1,\ldots, v_n)=(u\circ \pi)(v_1,\ldots, v_n)\]
where the rational map $\pi:M_{n \times n}(\symk) \dasharrow \grass_n(\symk)$ 
restricts to a morphism on an open subset of $M_{n \times n}(\symk)$ containing
the image of $u:\GL(n) \to M_{n \times n}(\symk)$.

Now, since $\hU$ is a group, the $(i,j)$ entry of the product of two elements is 
 the polynomial $p_{i,j}$ in the entries of the first row of the product; that is,
\[u(\beta_1,\ldots, \beta_n)u(\a_1,\ldots, \a_n)=u(\a_1\beta_1,\a_1^{\weight_2}\beta_2+\beta_1\a_2,
\ldots,\sum_{m=1}^n p_{m,n}(\a_1,\ldots \a_n)\beta_m)\]
for any $\a_1, \ldots, \a_n,\b_1, \ldots, \b_n$.
This implies that  
\[u(e_1,\ldots, e_n)\cdot u(\a_1,\ldots, \a_n)=u(\a_1e_1,\a_1^{\weight_2}e_2+\a_2e_1,
\ldots,\sum_{m=1}^n p_{m,n}(\a_1,\ldots \a_n)e_m)\]
where $\{e_1, \ldots, e_n\}$ is the standard basis for $\CC^n$. However, the $n$-tuple $$(\a_1e_1,\a_1^{\weight_2}e_2+\a_2e_1,
\ldots,\sum_{m=1}^n p_{m,n}(\a_1,\ldots \a_n)e_m)) \in (\CC^n)^{\oplus n}$$ on the right hand side forms the columns of the matrix $ 
u(\alpha_1,\ldots,\a_n)$, so we arrive at
\begin{equation}\label{groupproperty}
u(e_1,\ldots, e_n)\cdot u(\a_1,\ldots, \a_n)=u(u(\a_1,\ldots, \a_n)\cdot e_1,\ldots , u(\a_1,\ldots, \a_n)\cdot e_n).
\end{equation}
Since $u(\a_1,\ldots, \a_n)$ lies in the standard Borel subgroup $B_n$
of $\GL(n)$, the matrices $u(e_1,\ldots, e_n)$ and $u(e_1,\ldots, e_n)\cdot u(\a_1,\ldots, \a_n)$ represent the same element in $\grass_n(\symk)$; that is, 
in $\grass_n(\symk)$ we have
\begin{multline}\nonumber
\zdis_n=\pi(u(e_1,\ldots ,e_n))=\pi(u(e_1,\ldots, e_n)\cdot u(\a_1,\ldots, \a_n))=\\
\pi(u(u(\a_1,\ldots, \a_n)\cdot e_1,\ldots , u(\a_1,\ldots, \a_n)\cdot e_n)
\end{multline}
which completes the proof that $\hU \subseteq \GL(n)_{\zdis_n}$.
 
It remains to prove that $\GL(n)_{\zdis_n}\subseteq \hU$. 
Suppose that $g = (g_{ij})_{i,j=1}^n \in \GL(n)_{\zdis_n}$; we want to show that 
$g \in \hU$. For $1 \leq m \leq n$ let 
$$g^{\leq m} = (g_{ij})_{i,j=1}^m \in \GL(m)$$
be the upper left $m\times m$ block of $g$. 
Recall that by Remark \ref{rmk3.2} if $j \geq i \geq 2$ then $p_{i,j}(\a_1,\ldots,\a_n) = p_{i,j}(\a, \ldots, \a_{j-1})$
depends only on $\a_1,\ldots,\a_{j-1}$.
We will prove by induction on $m$
that
$$g^{\leq m} = u(g_{11},g_{12}, \ldots, g_{1m})$$
This is clear for $m=1$ since $g^{\leq 1} =(g_{11})=u(g_{11})$.
Suppose that it is true for some $m<n$. Since $g \in  \GL(n)_{\zdis_n}$ the
Pl\"{u}cker coordinates
$$ 
 e_1 \wedge (e_2+e_1^{\weight_2}) \wedge \ldots \wedge \sum_{i=1}^n p_{i,n}(e_1,\ldots, e_n)  $$
of $\zdis_n$ agree up to multiplication by a nonzero scalar  with the
Pl\"{u}cker coordinates
$$ 
 g e_1 \wedge (g e_2+g e_1^{\weight_2}) \wedge \ldots \wedge \sum_{i=1}^n p_{i,n}(g e_1,\ldots, g e_n)  $$
of $g \zdis_n$, where 
$g e_j = \sum_{s=1}^n g_{sj}e_s$ and 
$p_{i,j}(g e_1,\ldots,g e_n) \in \sym^{\weight_i}(\CC^n) \subseteq \symk$. By the inductive hypothesis we have
$$g_{ij} = p_{i,j}(g_{11}, \ldots, g_{1j})$$
for $1 \leq i \leq m$ and $1 \leq j \leq m$, so with our previous notation 
\[g^{\le m}=u(g_{11},\ldots, g_{1m})\in \hU\]
holds, and therefore $g^{\le m}$ fixes $\zdis_m$; thus \[\zdis_m=\pi(u(e_1,\ldots, e_m))=\pi(u(g^{\le m}e_1,\ldots ,g^{\le m}e_m)).\]
In coordinates this means that
$$ 
\cdot e_1 \wedge (e_2+e_1^{\weight_2}) \wedge \ldots \wedge \sum_{i=1}^m p_{i,m}(e_1,\ldots, e_m)$$
agrees up to multiplication by a nonzero scalar with  
$$ g e_1 \wedge (g e_2+g e_1^{\weight_2}) \wedge \ldots \wedge \sum_{i=1}^m p_{i,m}(g e_1,\ldots, g e_m).  $$
Therefore 
$$ 
 e_1 \wedge (e_2+e_1^{\weight_2}) \wedge \ldots \wedge \sum_{i=1}^m p_{i,m}(e_1,\ldots, e_m) \wedge  \sum_{i=1}^{m+1} p_{i,m+1}(e_1,\ldots, e_{m+1})
\wedge \ldots \wedge \sum_{i=1}^n p_{i,n}(e_1,\ldots, e_n)$$ and $$
 e_1 \wedge ( e_2+ e_1^{\weight_2})  \wedge \ldots \wedge \sum_{i=1}^m p_{i,m}(e_1,\ldots, e_m) \wedge  \sum_{i=1}^{m+1} p_{i,m+1}(ge_1,\ldots, g e_{m+1})\wedge \ldots \wedge \sum_{i=1}^n p_{i,n}(g e_1,\ldots, g e_n)$$
agree up to multiplication by a nonzero scalar.
Applying the identification  
\begin{equation}\label{wedge}
\bigwedge^n (\oplus_{i=1}^tV_i)=\bigoplus_{p_1+\ldots+p_t=n}\left(
\wedge^{p_1}V_1 \otimes \ldots \otimes \wedge^{p_t}V_t\right),
\end{equation}
with $V_1=\bigwedge^{m+1}(\CC^n \oplus \sym^{\weight_2}\CC^n \oplus \cdots\oplus \sym^{\weight_{m+1}}\CC^n)$ and $$V_2=\sym^{\weight_{m+2}}\CC^n,\ldots, V_{n-m}=\sym^{\weight_n}\CC^n$$ we get a natural $\GL(n)$-equivariant projection to the direct summand  corresponding to $p_1=m+1,p_2=\ldots =p_{n-m}=1$ given by 
$$\pi: \bigwedge^n \symk \to \bigwedge^{m+1}(\CC^n \oplus \sym^{\weight_2}\CC^n \oplus \cdots\oplus \sym^{\weight_{m+1}}\CC^n) \otimes \sym^{\weight_{m+2}} \otimes \cdots \otimes \sym^{\weight_n}\CC^n$$
which takes $e_1 \wedge (e_2+e_1^{\weight_2}) \wedge \ldots \wedge \sum_{i=1}^n p_{i,n}(e_1,\ldots, e_n)$ to
$$e_1 \wedge (e_2+e_1^{\weight_2}) \wedge \ldots \wedge \sum_{i=1}^m p_{i,m}(e_1,\ldots, e_m) \wedge \sum_{i=1}^{m+1} p_{i,m+1}(e_1,\ldots, e_{m+1}) \otimes e_1^{\weight_{m+2}} \otimes \cdots \otimes
e_1^{\weight_n}.$$
This must agree up to multiplication by a nonzero scalar with the projection
\begin{multline}\nonumber
\pi\left(g e_1 \wedge (g e_2+g e_1^{\weight_2}) \wedge \ldots \wedge \sum_{i=1}^m p_{i,m}(e_1,\ldots, e_m) \wedge \sum_{i=1}^n p_{i,n}(g e_1,\ldots, g e_n)\right)=\\
e_1 \wedge (e_2+e_1^{\weight_2}) \wedge \ldots \wedge \sum_{i=1}^{m+1} p_{i,m+1}(g e_1,\ldots, g e_{m+1}) \otimes {q_{m+2}} \otimes \cdots \otimes
{q_n}
\end{multline}
for some $q_j \in \sym^{\weight_j}\CC^n$ for ${m+2}\leq j \leq n$. It follows from this
that
\begin{multline}\label{transformed}
\lambda e_1 \wedge (e_2+e_1^{\weight_2}) \wedge \ldots \wedge \sum_{i=1}^{m+1} p_{i,m+1}(e_1,\ldots, e_{m+1})=\\
e_1 \wedge (e_2+e_1^{\weight_2}) \wedge \ldots \wedge \sum_{i=1}^{m+1} p_{i,m+1}(g e_1,\ldots, g e_{m+1}),
\end{multline}
for some nonzero scalar $\lambda$. 


Now, $g^{\le m}=u(g_{11},\ldots, g_{1m})$ and therefore by \eqref{groupproperty}
\[u(g^{\le m}e_1,\ldots, g^{\le m}e_m)=u(e_1,\ldots, e_n)\cdot u(g_{11},\ldots, g_{1m}).\]
But if $m+1 \geq i\ge 2$ then $p_{i,m+1}(\a_1,\ldots,\a_{m})$ is a polynomial in $\a_1,\ldots, \a_{m}$, and does not depend on $\a_{m+1},\ldots, \a_n$. Therefore 
\begin{multline}
p_{i,m+1}(ge_1,\ldots, ge_{m})=\sum_{s=2}^{n}p_{is}(e_1,\ldots, e_m) p_{s,m+1}
(g_{11},\ldots,g_{1,m+1})\text{ for } 2\le i \le m+1 
\end{multline}
and
\[p_{1,m+1}(ge_1,\ldots, ge_{m+1})=ge_{m+1}=\sum_{i=1}^n g_{i,m+1}e_i. \]
Substituting this into \eqref{transformed} we arrive at the equation
\begin{multline}\label{transformed2}
\lambda \cdot \left(e_1 \wedge (e_2+e_1^{\weight_2}) \wedge \ldots \wedge \sum_{i=1}^{m+1} p_{i,m+1}(e_1,\ldots, e_{m+1})\right)=\\
=e_1 \wedge (e_2+e_1^{\weight_2}) \wedge \ldots \wedge \left(\sum_{i=2}^{m+1} \sum_{s=1}^{n}p_{s,m+1}
(g_{11},\ldots,g_{1,m+1})p_{is}(e_1,\ldots, e_{m+1})+\sum_{s=2}^n g_{s,m+1}e_i\right).
\end{multline}

There is another $\GL(n)$-equivariant projection to the direct summand corresponding to $V_i=\sym^{\weight_i}\CC^n$ and $p_1=2,p_2=\ldots =p_{m}=1$ in \eqref{wedge}, given by
\[
\rho:\bigwedge^{m+1}(\CC^n \oplus \sym^{\weight_2}\CC^n \oplus \cdots\oplus \sym^{\weight_{m+1}}\CC^n) \to \\
 \wedge^2\CC^n \otimes \sym^{\weight_2}\CC^n \otimes \cdots \otimes \sym^{\weight_m}\CC^n
\]
which takes the left hand side of \eqref{transformed2} to 
\[\lambda (e_1\wedge e_{m+1}) \otimes  e_1^{\weight_2} \otimes  \ldots \otimes e_1^{\weight_m}\]
and the right hand side to 
\[\left(e_1 \wedge (\Sigma_{s=2}^m(p_{s,m+1}
(g_{11},\ldots,g_{1,m+1})-g_{s,m+1})e_s+g_{m+1,m+1}e_{m+1})\right) \otimes e_1^{\weight_2} \otimes \ldots \otimes e_1^{\weight_m} . \] 
These two are equal, so we obtain 
\begin{equation}\label{lambda1}
g_{s,m+1}=p_{s,m+1}(g_{11},\ldots, g_{1,m+1}) \text { for } s \neq 1,m+1
\,\,\, \mbox{ and } \,\,\,
\lambda=g_{m+1,m+1}.
\end{equation}
Note that the right hand side of \eqref{transformed2} is independent of $b_{1,m+1}$, which can be chosen arbitrarily, as we expect.  
Finally, for $s=m+1$, we take the third $\GL(n)$-equivariant projection corresponding to $V_i=\sym^{\weight_i}\CC^n$ and $p_1=\ldots =p_n=1$ in \eqref{wedge},
given by
\begin{multline}\nonumber
\xi:\bigwedge^{m+1}(\sym^{\weight_1}\CC^n \oplus \sym^{\weight_2}\CC^n \oplus \cdots\oplus \sym^{\weight_{m+1}}\CC^n) \to \\
 \CC^n \otimes \sym^{\weight_2}\CC^n \otimes \cdots \otimes \sym^{\weight_m}\CC^n \otimes \sym^{\weight_{m+1}}\CC^n,
\end{multline}
and project the equation \eqref{transformed2}. We get
\[\lambda\cdot e_1^{\weight_1} \otimes  e_1^{\weight_2} \otimes  \ldots \otimes e_1^{\weight_{m+1}}=e_1^{\weight_1} \otimes  e_1^{\weight_2} \otimes  \ldots \otimes e_1^{\weight_m} \otimes p_{m+1,m+1}(b_{11},\ldots ,b_{1,m+1})e_1^{\weight_{m+1}}\]
which gives $\lambda=p_{m+1,m+1}(b_{11},\ldots ,b_{1,m+1})$. From \eqref{lambda1} we get $$g_{m+1,m+1}=p_{m+1,m+1}(b_{11},\ldots ,b_{1,m+1})$$ and Theorem \ref{embed} is proved. 
\end{proof}

\subsection{Changing the basis of $\lieu$}

We observed in Proposition \ref{homogprop} 
that the left-right multiplication action of the subgroup $\CC^*$ of
 $\hU$ implies that the polynomial entry $p_{i,j}(\a)$ of an element of $\hU$
with parameters $\a$ in the first row has degree $i$ and weighted degree $\weight_j$
in $\a$. Similarly we have a bigrading on $\symk$ as follows: 
 the Lie algebra $\lieu=\mathrm{Lie}(U)$ decomposes into eigenspaces for the adjoint action of  $\Lie \CC^*=\CC z = \lieu_1$ as
\[\lieu=\oplus_{i=1}^r \lieu_i,\]
where $z \in \lieu_1 \setminus \{0\}$ and 
\[\lieu_i=\{x \in \lieu:[x,z]=(\tilde{\weight}_i-1) x\}\] 
if $\tilde{\weight}_1,\ldots, \tilde{\weight}_r$ are the different weights among $\weight_1,\ldots, \weight_n$.
This induces a decomposition 
\[\symk=\CC^n \oplus (\lieu_2 \otimes \sym^{\tilde{\weight}_2}\CC^n) \oplus \ldots \oplus (\lieu_r \otimes \sym^{\tilde{\weight}_r}\CC^n)\]
of $\symk$.
Let
$\sym^{a}_{b}\CC^n=\oplus_{\weight_{i_1}+\ldots +\weight_{i_a}=b} (\CC e_{i_1}\ldots e_{i_a})\subseteq \sym^{a}\CC^n$  
 and define
\[\symkd =\oplus_{i,j=1}^r (\mathfrak{u}_i \otimes \mathfrak{u}_j \otimes \sym^{\tilde{\weight}_i}_{\tilde{\weight}_j}\CC^n).\]
The image of the embedding $\phi_n$ of $\GL(n)/\hU$ sits in the subset
$\grass_n( \symkd)$ of $\grass_n( \symk)$, and the group 
\[\widetilde{\GL}(\lieu)=\CC^* \times \GL(\lieu_2) \times \ldots \times \GL(\lieu_r) \subset \GL(\hat{\mathfrak{u}})\]
acts on $\symkd$ via conjugation 
and thus on $\grass(n,\symkd)$. 
If $g \in \widetilde{\GL}(\lieu)$ 
then the subgroup 
\[ g^{-1}\hU g \]
of $\GL(n)$ with Lie subalgebra $g^{-1}\lieu g
$ has the same form as $\hU$ and so we can compare the corresponding
embeddings $\phi_n$ of $\GL(n)/\hU$ and $\GL(n)/g^{-1}\hU g$ in   $\grass(n,\symkd)$; let us denote these by $\phi_{\hU}$ and $\phi_{g^{-1}\hU g}$.
The linear forms in the first row of $g^{-1}\lieu g$ (and the same linear forms in the first row of $g^{-1} \hU g$) are linearly independent, and give parameters $b_1,\ldots, b_n$ for the group and its Lie algebra. The corresponding embedding is then $\phi_{g^{-1}\hU g}$, and we have 

\begin{prop}\label{changebasis}
A linear change of basis of $\hat{\lieu}$ by any element of $\widetilde{\GL}(\lieu)$ does not change the closure of the image of the embedding $\phi_{\hU}$ of
$\GL(n)/\hU$ into the Grassmannian $\grass(n,\symkd)$ up to isomorphism.
\end{prop}

\begin{proof}
This follows from the commutativity of the diagram
\begin{diagram}[LaTeXeqno,labelstyle=\textstyle]\label{diagram}
\GL(n)/\hU & \rInto^{\phi_{\hU}} & \grass(n,\symkd) \\
\dTo_{\mathrm{conj(g)}} & & \dTo_{\mathrm{conj}(g)\circ (g_{11}\cdot g^{-1})}\\
\GL(n)/g^{-1}\hU g & \rInto^{\phi_{g^{-1}\hU g}} & \grass(n,\symkd),
\end{diagram}
where 
\begin{enumerate}
\item  the left vertical $\mathrm{conj}(g)$ is the conjugation action sending the coset $\hU h\in \GL(n)/\hU$ to $(g^{-1}\hU g)(g^{-1}hg)=g^{-1}\hU hg \in \GL(n)/g^{-1}\hU g$; 
\item  the right vertical map is the composition of the multiplication by the scalar $g_{11}$ and the matrix $g^{-1}$ on $\CC^n$, and conjugation with $g \in \widetilde{\GL}(\lieu)$ on $\symk$.
\end{enumerate}
\end{proof}

\section{Singularities, jet differentials and curvilinear Hilbert schemes}\label{sec:jetdifferentials}

In this section we will
study an important example of a group of the form $\hU$ and its projective embedding $\phi_{\hU}:\GL(n)/\hU \hookrightarrow \grass_n(\symk)$ given by
Theorem \ref{embed} whose image
is contained in the affine open subset of the Grassmannian $\grass_n(\symk)$
where the coordinate corresponding to $\wedge^n \CC^n$ is nonzero.
We will see that here the codimension-$2$ property does not hold. Nonetheless in the next section we will see that a modification of this embedding can be used to find an affine embedding of $\SL(n)/U \rtimes F_{(M)}$ (where $U \rtimes F_{(M)}$ is an extension of $U$ by a finite subgroup of $\SL(n)$) for which the boundary does have codimension at least two.

The example we will study in this section is given by $\hU = \bbg_n \leq  \GL(n)$, where as in the introduction $\bbg_n$ is the group of polynomial reparametrisations of $n$-jets of holomorphic germs $(\CC,0) \to (\CC,0)$. This group plays a central role in global singularity theory \cite{arnold} and in the recent history of hyperbolic varieties \cite{demailly,dmr,kobayashi, siu}. We will see that the
compactification $\overline{\GL(n)/\GG_n}$
constructed in $\S$4 as the closure of an orbit of $\GL(n)$ with stabiliser
$\bbg_n$  in a Grassmannian $\grass_n(\symk)$ is isomorphic to the so-called curvilinear component of the punctual Hilbert scheme on $\CC^n$ \cite{b2,bertin}. 

\subsection{Singularity theory in a nutshell \cite{arnold,b2,bsz,gaffney,mather,porteous,ronga}}
Let $J_n(m,l)$ denote the space of $n$-jets of holomorphic map germs from $\CC^m$ to $\CC^l$ mapping the origin to the origin. This is a finite dimensional complex vector space, and there is a complex linear composition 
of jets
\[J_n(m,l) \otimes J_n(l,p) \to J_n(m,p).\]
Let $J_n^\reg(m,l)$ denote the open dense subset of $J_n(m,l)$ consisting of jets whose linear part is regular (that is, of maximal rank). Note that
$$\bbg_n=J_n^{\reg}(1,1)$$ 
becomes a group under composition of jets, and it acts via reparametrisation on $J_n(1,n)$.  

If $z$ denotes the standard complex coordinate on $\CC$, then elements of the vector space $J_n(1,1)$ can be identified with polynomials of the form $p(z)=\a_1z+\ldots +\a_nz^n$ with coefficients in $\CC$, so $\{z,z^2,\ldots, z^n\}$ is a natural basis for $J_n(1,1)$ over $\CC$. The composition 
of $p(z)$ with $q(z)=\b_1z+\ldots +\b_nz^n$ is 
\[(p \circ q)(z)=(\a_1\b_1)z+(\a_2\b_1+\a_1^2\b_2)z^2+\ldots \]
which corresponds (with respect to the basis $\{z,z^2,\ldots, z^n\}$) to multiplication on the right by the matrix
\begin{equation}\label{bbg}
\left(
\begin{array}{ccccc}
\alpha_1 & \alpha_2   & \alpha_3          & \ldots & \alpha_n\\
0        & \alpha_1^2 & 2\alpha_1\alpha_2 & \ldots & 2\alpha_1\alpha_{n-1}+\ldots \\
0        & 0          & \alpha_1^3        & \ldots & 3\alpha_1^2\alpha_ {n-2}+ \ldots \\
0        & 0          & 0                 & \ldots & \cdot \\
\cdot    & \cdot   & \cdot    & \ldots & \alpha_1^n
\end{array}
 \right)
\end{equation}
where the polynomial in the $(i,j)$ entry is
\[p_{i,j}({\alpha}_1, \ldots, \alpha_n)=\sum_{\ell_1+\ell_2+\ldots +\ell_i=j}\alpha_{\ell_1}\alpha_{\ell_2} \ldots \alpha_{\ell_i}.\]
Thus the subgroup $\bbg_n$ of $\GL(n)$ is an extension by $\CC^*$ of its
unipotent radical $\mathbb{U}_n$, and both $\bbg_n$ and $\mathbb{U}_n$
are generated along the first row and have the form (\ref{presentation})
with weights $1,2,\ldots,n$.
We can think of the quotient $J^n(1,n)/\bbg_n$ as the moduli space of $n$-jets of entire holomorphic curves in $\CC^n$. 

Global singularity theory studies global and local behavior of singularities of holomorphic maps between complex manifolds; \cite{arnold} is a standard reference. For a holomorphic map $f:M \to N$ with $f(p)=q \in N$ the local algebra is $A(f)=\mathfrak{m}_p /f^*\mathfrak{m}_q$; if $\mathfrak{m}_p$ is a finite $\mathfrak{m}_q$-module, then $p$ is an isolated singularity. 
For a complex nilpotent algebra $A$ with $\dim_{\CC} A=n$ we define
\[\Sigma_A(m,l)=\left\{f \in J_n(m,l):A(f)\simeq A \right\}\]
to be the subset of $J_n(m,l)$ consisting of germs with local algebra at the origin isomorphic to $A$; these are known as the $A$-singularity germs. There is a natural hierarchy of singularities where for two algebras $A$ and $A'$ of the same dimension $n$ we have 
\[A>A' \text{ if } \Sigma_{A}(m,l) \subset \overline{\Sigma_{A'}(m,l)} \text{ for } l>>m.\]
When $A_n=z\CC[z]/z^{n+1}$ is the nilpotent algebra generated by one element, the corresponding singularities are the so-called $A_n$-singularities (also known as Morin singularities or curvilinear singularities). These vanish to order $n$ in some direction, giving us the geometric description 
\[\Sigma_{A_n}(m,l)=\{\psi\in J_n(m,l): \exists \gamma \in J_n(1,m) \text{ such that } \gamma \circ \psi=0\}.\] 
If $\psi\in J_n(m,l)$ and a test curve $\gamma_0 \in J_n(1,m)$ exists with
$  \gamma_0 \circ \psi=0 $, then there is a whole family of such test curves. Indeed, for any $\beta \in J_n^\reg(1,1)$, the curve $\beta \circ \gamma_0$
is also a test curve, and in fact if $\psi \in J_n^{\reg}(m,l)$ then we get all test curves $\gamma \in J_n(1,m)$ with
$  \gamma \circ \psi=0 $ in this way. This description of the curvilinear jets using the so-called \lq test-curve model' goes back to Porteous, Ronga and Gaffney \cite{gaffney,porteous,ronga}.

This means that the regular part of $\Sigma_{A_n}(m,l)$ fibres over the quotient $J_n^\reg(1,m)/\bbg_n$, which can be thought of as representing moduli of $n$-jets of holomorphic germs in $\CC^m$. We can identify $J_n(1,m)$ with the set $M_{m \times n}(\CC)$ of $m\times n$ complex matrices by putting the $i$th derivative of $\gamma \in J_n(1,m)$ into the $i$th column of the corresponding matrix, and then $J_n^\reg(1,m)$ consists of the matrices in $M_{m\times n}(\CC)$ with nonzero first column. Therefore when $m=n$ the quotient
$J^{\rm{reg}}_n(1,n)/\bbg_n$ contains  $\GL(n)/\bbg_n$ as a dense open subset. 

In \cite{bsz} the first author and Szenes use this model of the Morin singularities and 
the machinery of equivariant localization to compute some useful invariants  of $A_n$ singularities: their Thom polynomials. These ideas were later generalised in \cite{kazarian,rf}. 

The hierarchy of singularities is only partially understood, but there are well-known singularity classes in the closure of the $A_n$-singularities (for details see \cite{arnold,rimanyi}). In particular, for $n=4$, the so called $I_{a,b}$ singularities with $a+b=4$ are defined by the algebra
\[A_{I_{a,b}}=(x,y)/(xy,x^a+y^b)\]     
and it is well known (see \cite{rimanyi,rf}) that 
\[\Sigma_{I_{2,2}}(m,l) \subset \overline{\Sigma_{A_4}(m,l)}\]
has codimension $1$ in $\overline{\Sigma_{A_4}(m,l)}$. But as we have just seen, a dense open subset of $\Sigma_{A_4}(4,l)$ fibres over $\GL(4)/\bbg_4$, and the latter is embedded via $\phi_4$ (see Corollary \ref{phi}) into $\grass_4(\symk)$ where $\bk=(1,2,3,4)$ as at \eqref{bbg}. When $l=1$, then in fact 
\[\overline{\Sigma_{A_4}(4,1)}=\overline{\phi_4(\GL(4)} \subseteq \grass_4(\symk),\]
because the fibres are trivial. So it follows that $\Sigma_{I_{2,2}}(4,1)$ lies in the boundary of $\overline{\phi_4(\GL(4)}$ 
and 
has codimension one.
In fact
\[\zdis_{2,2}=\lim_{t \to 0}\left(\begin{array}{cccc}t & t^{-2} & -t^{-5} & 0\\ 0 & 1 & -2t^{-3} & 0\\ 0 & 0 & t^{-1} & 0\\ 0 & 0 & 0 & 1  \end{array}\right)\cdot \zdis_n=e_1 \wedge e_2 \wedge (e_3+e_1^2) \wedge (e_4+e_1e_3+e_2^2+e_1^3)\]
sits in $\Sigma_{I_{2,2}}(4,1)$ and its orbit has codimension $1$ in $\overline{\phi(\GL(4)}$. Indeed it can be checked by direct computation that the stabiliser of $\zdis_{2,2}$ is 
\[ \left\{ \left(\begin{array}{cccc}t & a & b & c\\ 0 & t^{3/2} & -2t^{1/2}a & d\\ 0 & 0 & t^{2} & tb+a^2\\ 0 & 0 & 0 & t^3  \end{array}\right): t \in \CC^*, a,b,c,d \in \CC \right\} \]
which has dimension $5$, whereas the stabiliser $\GG_4$ of $\zdis_4$ in $\GL(4)$ has dimension $4$.  
 
\subsection{Invariant jet differentials and the Demailly bundle}

Jet differentials have played a central role in the study of hyperbolic varieties. Their contribution can be traced back to the work of Bloch \cite{bloch}, Cartan \cite{cartan}, Ahlfors \cite{ahlfors}, Green and Griffiths \cite{gg}, Siu \cite{siu}, whose ideas were extended in the seminal paper of Demailly \cite{demailly}, and recently used by Diverio, Merker and Rousseau \cite{dmr} and the first author in \cite{b} to prove the Green Griffiths conjecture for generic projective hypersurfaces of high order; see also the survey papers \cite{kobayashi,demailly,dr} for more details.

Let 
\[f:\CC \to X,\ \ t\to f(t)=(f_1(t),f_2(t), \ldots ,f_d(t))\]
be a curve written in local holomorphic coordinates
$(z_1,\ldots ,z_d)$ on a complex manifold $X$, where $d=\dim(X)$. Let $J_n(X)$ be the $n$-jet bundle over $X$ of holomorphic curves, whose fibre $(J_n(X))_x$
at $x \in X$ is 
the space of $n$-jets of germs at $x$ of holomorphic curves in $X$.
 This fibre can be identified with $J_n(1,d)$. The group of reparametrisations $\bbg_n=J_n^{\reg}(1,1)$
acts fibrewise on $J_n(X)$, and the action is linearised as at \eqref{bbg}. For $\lambda \in \CC^*$ we have 
\[(\lambda \cdot f)(t)=f(\lambda \cdot t),\text{ so } \lambda \cdot (f',f'',\ldots ,f^{(k)})=(\lambda f',\lambda^2 f'',\ldots ,\lambda^k f^{(k)}).\] 
Polynomial functions on $J_n(X)$ correspond to algebraic differential operators called jet differentials; these have the form 
\[Q(f',f'',\ldots ,f^{(k)})=\sum_{\alpha_i \in \mathbb{N}^n}a_{\alpha_1,\alpha_2,\ldots \alpha_k}(f(t))(f'(t)^{\alpha_1}f''(t)^{\alpha_2}\cdots f^{(n)}(t)^{\alpha_n}),\]
where $a_{\alpha_1,\alpha_2,\ldots \alpha_n}(z)$ are holomorphic coefficients on $X$ and $t \mapsto f(t)$  is the germ of a holomorphic curve in $X$.
Here 
$Q$ is homogeneous of weighted degree $m$ under the $\CC^*$ action if and only if 
\[Q(\lambda f',\lambda^2 f'',\ldots ,\lambda^k f^{(n)})=\lambda^m Q(f',f'',\ldots ,f^{(n)})\]
for every $\lambda \in \CC$.
\begin{defn}
(i) (Green-Griffiths \cite{gg}) Let $E_{n,m}^{GG}$ denote the sheaf on $X$ of jet differentials of order $n$ and weighted degree $m$.

(ii) (Demailly, \cite{demailly}) The bundle of invariant jet differentials of order $n$ and weighted
degree $m$ is the subbundle $E_{n,m}$ of $E_{n,m}^{GG}$ whose elements are
invariant under the action of the unipotent radical $\mathbb{U}_n$ of the reparametrisation group $\bbg_n$ and transform under the action of $\bbg_n$ as 
\[Q((f\circ \phi)',(f \circ \phi)'',\ldots ,(f \circ \phi)^{(n)})=\phi'(0)^mQ(f',f'',\ldots, f^{(n)})\]
for $\phi \in \bbg_n$.
\end{defn}
Thus the fibres of the Demailly bundle $\bigoplus_{m \geqslant 0} E_{n,m}$ are isomorphic to $\CC[J_n(1,d)]^{\mathbb{U}_n}$, where $\mathbb{U}_n$ is the unipotent radical of $\bbg_n$. Demailly in \cite{demailly} conjectured that this algebra of invariant jet differentials is finitely generated. 
%
%
Rousseau (\cite{rousseau}) and Merker (\cite{Merker1,Merker2}) showed that when both $n$ and $\dim X$ are small then this conjecture is true, and in \cite{Merker2} Merker provided an algorithm which produces  finite sets of generators when they exist for any $\dim X$ and $n$. 
In \cite{BK} the authors put forward a proof  that $\UU_n$ is a Grosshans subgroup of $\SL(n)$, with the Demailly conjecture as an immediate corollary, but we later discovered a gap in that proof. 
In this paper we are studying quotient constructions for linear actions such as that of $\UU_n$ on a fibre of the Demailly bundle $\bigoplus_{m \geqslant 0} E_{n,m}$ from a more geometric point of view; however
it will follow from this point of view (see Remark \ref{Demconj}) that the subalgebra of $\CC[J_n(1,d)]^{\mathbb{U}_n}$ spanned by the jet differentials which are weight vectors with non-positive weight for the action of $\CC^* \leq \tilde{\GG}_n$ twisted by a well adapted rational character is finitely generated (cf. \cite{Merker1,Merker2}).

\subsection{Curvilinear Hilbert schemes}\label{subsec:hilb}

In \cite{b3}  the closure $\overline{J_n(1,d)/\bbg_n}$ 
of ${J_n(1,d)/\bbg_n}$ embedded in $\grass_n(\oplus_{i=1}^n \Sym^i\CC^d)$  is identified with the curvilinear component of the $n+1$-point punctual Hilbert scheme on $\CC^d$; this geometric component of the punctual Hilbert scheme on $\CC^d$ is thus the compactification of a non-reductive quotient.

 Hilbert schemes of points on surfaces form a central object of geometry and representation theory and have a rich literature (see for example \cite{nakajima,bertin}). Recently many interesting connections between Hilbert schemes of points on  planar curve singularities and the topology of their links have been discovered \cite{shende,oblomkovshende,ors,maulik}.  However, much less is known about Hilbert schemes or punctual Hilbert schemes on higher dimensional manifolds. 

As above let  $\bbg_n=\jetreg 11$ denote the group of $n$-jets of reparametrisation germs of $\CC$, which acts on the space $\jetreg 1d$ of $n$-jets of germs of curves $f:(\CC,0) \to (\CC^d,0)$ with nonzero linear part. 
As in $\S$4 we have a map 
\[\phi: \jetreg 1d \to \grass_n(\oplus_{i=1}^n \Sym^i\CC^d)\]
\[(v_1,\ldots, v_n)  \mapsto [v_1\wedge (v_2 + v_1^2)\wedge \ldots \wedge (\sum_{a_1+a_2+\ldots +a_i=n}v_{a_1}v_{a_2} \ldots v_{a_i})]\]
where $v_i \in \CC^d$ is the degree $i$ part of the germ in $\jetreg 1d$, so that $v_1\neq 0$.
This map is invariant under the action of  $\bbg_n = \jetreg 11$ on the left, and gives us an embedding 
\[\jetreg 1d/\bbg_n \hookrightarrow \grass_n(\oplus_{i=1}^n \Sym^i\CC^d)).\]
Let $X_{n,d} = \overline{\jetreg 1d/\bbg_n}$ denote the closure of the image of this embedding.

In \cite{b3} it is proved that $X_{n,d}$ is the curvilinear component of the punctual Hilbert scheme of $n+1$ points on $\CC^d$. This component is defined as follows. Let $(\CC^d)^{[n]}$ denote the Hilbert scheme of $n$ points on $\CC^d$; that is, the set of zero-dimensional subschemes of $\CC^d$ of length $n$. The punctual Hilbert scheme $(\CC^d)^{[n]}_0$ consists of those subschemes which are supported at the origin in $\CC^d$. The components of the punctual Hilbert scheme are not known for $d\ge 3$ but there is a distinguished component containing all curvilinear subschemes.      

\begin{defn} A subscheme $\xi \in (\CC^d)^{[n]}_0$ is called curvilinear if $\xi$ is contained in some smooth curve $C\subset \CC^d$. Equivalently, one might say that $\calo_\xi$ is isomorphic  to the $\CC$-algebra $\CC[z]/z^{n}$.
The punctual curvilinear locus is the set of curvilinear subschemes supported at the origin in $\CC^d$ and its closure $\curv^{[n]}_d$ is the (punctual) curvilinear component of $(\CC^d)^{[n]}_0$. 
\end{defn}

Let $\mathfrak{m}=(x_1,\ldots, x_d)\subset \calo_{\CC^d,0}$ denote the maximal ideal of the local ring at the origin. Then 
\[\curv^{[n]}_d=\overline{\{I \subset \mathfrak{m}:\mathfrak{m}/I \simeq t\CC[t]/t^{n}\}}.\]
Note that $\Sym^{\le n} \CC^d=\mathfrak{m}/\mathfrak{m}^{n+1}=\oplus_{i=1}^n
\Sym^i\CC^d$ consists of function-germs of degree $\le n$, and the punctual Hilbert scheme sits naturally in its Grassmannian
\[\rho:(\CC^d)^{[n+1]}_0  \hookrightarrow \grass(n,\Sym^{\le n} \CC^d)\]
\[I \mapsto \mathfrak{m}/I. \]
The idea of \cite{b3} to describe the curvilinear component is the observation that curvilinear subschemes have test curves; that is, map germs $\gamma \in J_n(1,d)$ on which they vanish up to order $n$, so that $\gamma(\CC) \subseteq \mathrm{Spec}(\mathfrak{m}/I)$. Such a test curve is unique up to polynomial reparametrisation of $(\CC,0)$.  Therefore the image of $\phi$ is the same as the image of $\rho$ and their closures coincide.

\begin{prop} For $d,n \in \mathbb{Z}^{>0}$ we have $\curv^{[n+1]}_d=X_{n,d}$.
\end{prop}
When $d=2$ the curvilinear component $\curv^{[n+1]}_2$ is dense in $(\CC^2)^{[n+1]}_0$, and therefore the full punctual Hilbert scheme is equal to the closure of the image of $\phi$.
\begin{corollary} $(\CC^2)^{[n+1]}_0=X_{n,2}$ for any positive integer $n$.
\end{corollary}

This description of the curvilinear component becomes particularly useful when
$n \le d$ so that the number of points is not more than the dimension $d$ plus $1$. In this case, the curvilinear component $\curv^{[n+1]}_d$ is the closure of a $\GL(n)$-orbit in the Grassmannian $\grass_n(\Sym^{\le n}\CC^n)$. In fact, for any fixed basis  $\{e_1,\ldots ,e_d\}$ of $\CC^d$, we have $X_{n,d}=\overline{\mathrm{GL}(n) \cdot \mathbf{e}_{n,d}}$
where 
\[\mathbf{e}_{n,d}=e_1 \wedge (e_2\oplus e_1^2) \wedge \ldots \wedge (\sum_{\substack{a_1+\ldots+a_l=n\\ l\le d}} e_{a_1}\ldots e_{a_l}).\] 
This follows when $n \le d$ from the fact that $\phi$ is $\GL(n)$-equivariant, but for $n>d$ it cannot be true as the dimension of the quotient is larger than the dimension of $\GL(n)$. In particular, when $d=n$ we have $\mathrm{GL}(n) \subset J_n^\reg(1,n)$, and an embedding
\[\mathrm{GL}(n)/\bbg_n \subseteq \grass_n(\Sym^{\le n}\CC^n)\]
and the closure of the image $X_{n,n}=\curv^{[n+1]}_n$ is the curvilinear component of the punctual Hilbert scheme of $n+1$ points on $\CC^n$. In \cite{b3} this parametrisation of the curvilinear Hilbert scheme is used to develop an iterated residue formula for cohomological intersection numbers of tautological bundles over the curvilinear component. 

\section{Proof of the 
 theorems}

\subsection{Boundary components of $\GL(n)/\hU$ in $\PP(\wsymk n)$}\label{affineboundary}


Let us now return to the situation in $\S$4 where $\hat{U}$ and $U$ are subgroups of $\GL(n)$ of the form described at (\ref{presentation}). 
In \S \ref{sec:construction} we embedded $\GL(n)/\hU$ in the Grassmannian $$\grass_n(\symk) \subseteq \PP(\wsymk n)$$ as the $\GL(n)$ orbit of 
\[\zdis_n=\phi_n(e_1,\ldots, e_n)=[e_1\wedge (e_2 + e_1^{\weight_2})\wedge \ldots \wedge (\sum_{i=1}^n p_{in}(e_1,\ldots, e_n))]\in \PP[\wsymk n],\]
and observed at Remark \ref{afemb} that the image of this embedding lies in the open affine subset defined by the
non-vanishing of the
coordinate in $ \PP(\wsymk n)$ corresponding to the one-dimensional summand
$\wedge^n \CC^n$ of $\wsymk n$ spanned by $e_1 \wedge \cdots \wedge e_n$.  
In $\S$5 we saw that there exist examples where the image has codimension-one boundary components which meet this affine open subset, and therefore the Grosshans principle 
 is not applicable in this situation.

In this section we study first the boundary of  the orbit $
\GL(n)\zdis_n $ in the
affine space  $\mathcal{W}=\wsymk n$. The stabiliser of $\zdis_n$ in $\GL(n)$ is $U$. Let $\mathcal{W}_{v_1} $ be the linear subspace 
$$\mathcal{W}_{v_1} = \bigoplus_{(k_1,k_2,\ldots,k_n) \neq (1,1,\ldots,1)}
\wedge^{k_1} (\CC^n) \otimes \wedge^{k_2} (\sym^{\weight_2} \CC^n) \otimes \cdots \otimes  \wedge^{k_n} (\sym^{\weight_n} \CC^n)   $$
of $ \mathcal{W}$
where the coefficients corresponding to $v_1 \wedge v_1^{\weight_2} \wedge \ldots \wedge v_1^{\weight_n}$ are zero; that is, if $\pi^{\wedge}:\mathcal{W} \to \CC^n \wedge \Sym^{\weight_2} \CC^n \wedge \ldots \wedge \Sym^{\weight_n} \CC^n$ denotes the projection onto the corresponding summand of $ \mathcal{W}$ then 
\[\mathcal{W}_{v_1}=\{w\in \mathcal{W}: \pi^{\wedge}(w)=0\} \subset \mathcal{W}.\] 
Similarly, let 
\[\mathcal{W}_{\det}=\{w\in \mathcal{W}: \pi^{\det}(w)=0\} \subset \mathcal{W}\]
denote the kernel of the coordinate corresponding to $v_1\wedge \ldots \wedge v_n$, or equivalently the projection $\pi^{\det}:\mathcal{W} \to \wedge^n \CC^n$. 
\begin{prop}\label{boundary}
The boundary of the orbit $\GL(n)(\zdis_n)$ in $\mathcal{W}$ is contained in the union of the subspaces $\mathcal{W}_{v_1}$ and $\mathcal{W}_{\det}$:
\[\overline{\GL(n)(\zdis_n)} \setminus \GL(n)(\zdis_n )\subset \mathcal{W}_{v_1} \cup \mathcal{W}_{\det}\]
\end{prop}

\begin{proof}
Let $B_n \subset \GL(n)$ denote the standard upper triangular Borel subgroup of
$\GL(n)$ which stabilises the  
+filtration $\CC e_1 \subset \CC e_1 \oplus \CC e_2 \subset \cdots \subset \CC^n$.
Since $\GL(n)/B_n$ is projective we have
\[\overline{\GL(n)\cdot (\zdis_n \oplus e_1^r)}= \GL(n) \overline{B_n\cdot (\zdis_n \oplus e_1^r)} 
.\]
Let 
$$w=\lim_{m \to \infty} b^{(m)}(\zdis_n \oplus e_1^r) \in \overline{B_n(\zdis_n \oplus e_1^r)} 
\subseteq \mathcal{W}$$
be a limit point where  
\begin{equation}\label{bmform}
b^{(m)}=\left(\begin{array}{cccc}b^{(m)}_{11} & b_{12}^{(m)} & \ldots & b_{1n}^{(m)}  \\ 0 & b^{(m)}_{22} & \ldots & b^{(m)}_{2n} \\  & & \ddots 
&  \\ 0 & 0 & \ldots & b^{(m)}_{nn}\end{array}\right)\in B_{n} \subset \GL(n)
\end{equation}
 
Now expanding the wedge product in the definition of 
$\zdis_n$ we get
\[b^{(m)}(\zdis_n)=(\det(b^{(m)}) e_1 \wedge \ldots \wedge e_n+\ldots +(b^{(m)}_{11})^{1+\weight_2+\ldots +\weight_n}e_1 \wedge e_1^{\weight_2} \wedge \ldots \wedge e_1^{\weight_n})\]
so by considering the coefficient of $e_1 \wedge \ldots \wedge e_n$ we see that the determinant $\det(b^{(m)})$ tends to a limit in $\CC$
as $m \to \infty$. If this limit is zero then the limit point $w$ sits in $\mathcal{W}_{\det}$, so we will focus on the other case when 
$\lim_{m\to \infty}\det(b^{(m)}) \in \CC \setminus \{0\}$. Then we have to show that if $w$ is a boundary point then $w\in \mathcal{W}_{v_1}$, that is, $\lim_{m\to \infty} b_{11}^{(m)}=0$. 


We show indirectly that  $b^{(\infty)}_{11}=\lim_{m\to \infty} b^{(m)}_{11}\in \CC \setminus \{0\}$ implies that $w \in B_n (\zdis_n \oplus e_1)$ sits in the orbit. 
Here
\begin{multline}
b^{(m)}\zdis_n = b^{(m)}_{11}e_1 \wedge (b^{(m)}_{22}e_2+(b^{(m)}_{11})^{\weight_2}e_1^{\weight_2}) \wedge \ldots 
\wedge 
(b^{(m)}_{nn}e_n+b^{(m)}_{n-1n}e_{n-1}+\ldots +b^{(m)}_{1n}e_1+ \nonumber \\ 
 +\sum_{s=2}^{n-1} p_{sn}( 
b^{(m)}_{11}e_{1}, b^{(m)}_{22}e_{2}+b^{(m)}_{12}e_{1}, \ldots ,b^{(m)}_{nn} e_n+
\ldots +b^{(m)}_{1n}e_{1})+(b^{(m)}_{11})^{\weight_i}e_1^{\weight_i} ). 
\end{multline}
Now look at the coefficient of 
\[e_1 \wedge e_1^{\weight_2} \wedge \ldots \wedge e_1^{\weight_{i-1}} \wedge e_j \wedge e_1^{\weight_{i+1}} \wedge \ldots \wedge e_1^{\weight_n}\]
in $b^{(m)}(\zdis_n)$ when $1 \leq j \leq i \leq n$; we see that 
\[(b^{(m)}_{11})^{1+\weight_2+\ldots +\weight_{i-1} + \weight_{i+1} +\ldots +\weight_n}b^{(m)}_{ji} \]
tends to a limit in $\CC$ as $m \to \infty$, and 
so since $b^{(\infty)}_{11} \neq 0$
\[b^{(m)}_{ji} \to b^{(\infty)}_{ji} \in \CC.\]
Also
\[\lim_{m\to \infty} \det (b^{(m)})=b^{(\infty)}_{11} b^{(\infty)}_{22} \cdots b^{(\infty)}_{nn} \in \CC \setminus \{0\},\]
so $b^{(m)} \to b^{(\infty)} \in \GL(n)$. Therefore $w = b^{(\infty)}(\zdis_n \oplus e_1)$
lies in the orbit $\GL(n)(\zdis_{n} \oplus e_1)$ as required.

\end{proof}

\begin{corollary} \label{corbound}
The boundary of the orbit $\GL(n)[\zdis_n]$ in $\PP(\mathcal{W})$ is contained in the union of the subspaces $\PP(\mathcal{W}_{v_1})$ and $\PP(\mathcal{W}_{\det})$.
\end{corollary}

\begin{proof}
By rescaling using elements of $\CC^*=\hU/U$ we can assume that 
\[\lim_{m\to \infty} b^{(m)}[\zdis_n]=[\lim_{m\to \infty} b^{(m)}\zdis_n].\]
Proposition \ref{boundary} then gives us the statement. 
\end{proof}



\subsection{Well-adapted characters}

Let $X$ be a nonsingular complex projective variety on which $\hat{U}$ acts linearly with respect to a very ample line bundle $L$ inducing a $\hU$-equivariant embedding of $X$ in $\PP^N$. 
Let $\GL(n) \times_{\hU} X$ denote the quotient of $\GL(n) \times X$
by the free action of $\hU$ defined by $\hat{u}(g,x)=(g \hat{u}^{-1}, \hat{u}x)$ for $\hat{u} \in \hU$,
which is a quasi-projective variety by \cite{PopVin} Theorem 4.19. Then there
is an induced $\GL(n)$-action on $\GL(n) \times_{\hU} X$ given by left
multiplication of $\GL(n)$ on itself.
In cases where the action of $\hU$ on $X$ extends to an action of $\GL(n)$ there is an isomorphism of
$\GL(n)$-varieties 
\begin{equation} \label{basic} \GL(n) \times_{\hU} X \cong (\GL(n)/\hU) \times X
\end{equation} given by
$ [g,x] \mapsto (g\hU, gx). $
In this case the linearisation $L$ on $X$ extends to a very ample $\GL(n)$-linearisation 
$L^{(p,q)}$ on $\GL(n) \times_{\hU} X$ and its closure $\overline{\GL(n) \times_{\hU} X}$
using the inclusions
\[ \GL(n) \times_{\hU} X \hookrightarrow \GL(n) \times_{\hU} \mathbb{P}^N  \cong
(\GL(n)/{\hU}) \times \mathbb{P}^N \hookrightarrow \PP(\wsymk n) \times \PP^N\]
and  the very ample line bundle $\mathcal{O}_{\PP(\wsymk n)}(p) \otimes \mathcal{O}_{\PP^N}(q)$. Here 
the $\GL(n)$-invariants on $\GL(n) \times_{\hU} X$ are given by
\begin{equation} \label{name}  \bigoplus_{m \geq 0} H^0(\GL(n) \times_{\hU} X, L^{\otimes pm})^{\GL(n)} \cong
\bigoplus_{m \geq 0} H^0(X, L^{\otimes pm})^{\hU} = {\hat{\calo}}_{L^{\otimes p}}(X)^{\hU}.\end{equation}

Note that the normaliser $N_{\GL(n)}(\hU)$ of $\hU$ in $\GL(n)$ acts on the right on $\GL(n) \times_{\hU} X$ via
$$n[g,x] = [gn,n^{-1}x].$$
The central one-parameter subgroup $Z\GL(n)$ of $\GL(n)$ normalises $\hU$, and since $gn=ng$ for every $n \in Z\GL(n)$ and $g \in \GL(n)$, the right action of $Z\GL(n)$ on $\GL(n) \times_{\hU} X$ extends to a linear action on 
$\PP(\wsymk n) \times \PP^N$ given by $n(y,x) = (ny,x)$. Note also that the induced right action of $\hU$ on $\GL(n) \times_{\hU} X$ is trivial and its closure in $\PP(\wsymk n) \times \PP^N$ is trivial, and that the image of $Z\GL(n)$ in
$N_{\GL(n)}(\hU)/\hU$ is the same as the image of the one-parameter subgroup $\CC^*$ of $\tilde{U}$. However the induced right action of $\hU$ on the line bundle $\mathcal{O}_{\PP(\wsymk n)}(p) \otimes \mathcal{O}_{\PP^N}(q)$ is not trivial; it is multiplication by $(\weight_1 + \weight_2 + \cdots \weight_n)$ times the character $\hU \to \CC^*$ with kernel $U$.
Thus the weights $w$ and $\tilde{w}$ of the right actions of $Z\GL(n)$ and the one-parameter subgroup $\CC^* \leq \tilde{U}$ are related by
$$\tilde{w} = (1 + \weight_2 + \cdots + \weight_n)(w - n)$$
when we choose the basis vector $\mbox{diag}(1,1,\ldots,1)$ for $\mbox{Lie}Z\GL(n)$ and the basis vector
$$\mbox{diag}( 1 + \weight_2 + \cdots + \weight_n - n, 1 + \weight_2 + \cdots + \weight_n - n\weight_2, \ldots , 1 + \weight_2 + \cdots + \weight_n - n\weight_n)$$
for the Lie algebra of $\CC^* \leq \tilde{U}$.


When $\wsymk n$ is identified with the sum of summands
$$ \wedge^{k_1} (\CC^n) \otimes \wedge^{k_2} (\sym^{\weight_2} \CC^n) \otimes \cdots \otimes  \wedge^{k_n} (\sym^{\weight_n} \CC^n)   $$
over non-negative integers $k_1,\ldots, k_n$ such that $k_1 + \cdots + k_n = n$, the weight of the $Z\GL(n)$ action 
on the summand
$ \wedge^{k_1} (\CC^n) \otimes \wedge^{k_2} (\sym^{\weight_2} \CC^n) \otimes \cdots \otimes  \wedge^{k_n} (\sym^{\weight_n} \CC^n) $ is 
$$k_1\weight_1 + \ldots + k_n \weight_n.$$ 
Thus the weight of the right action of the one-parameter subgroup $\CC^* \leq \tilde{U}$ on this summand is
$$(k_1\weight_1 + \ldots + k_n \weight_n - n)(\weight_1 + \weight_2 + \cdots + \weight_n).$$ 
The weights for the $Z\GL(n)$ action on $\overline{\GL(n)/\hU}$ satisfy $ k_j + k_{j+1} + \cdots + k_n \leq n-j+1$ for $1 \leq j \leq n$ and therefore
\[\weight_{\min}=n\weight_1 \le k_1\weight_1 + \ldots + k_n \weight_n \le \weight_{\max}=\weight_1+\ldots +\weight_n\] 
where the minimum weight $\weight_{\min}=n\weight_1=n$ is taken on the summand spanned by $e_1\wedge \ldots \wedge e_n$ whereas the maximum weight $\weight_{\max}=\weight_1+\ldots +\weight_n$ is taken on the summand spanned by $v_1^{\weight _1} \wedge \ldots \wedge v_1^{\weight_n}$. In fact, since since $1 = \weight_1 < \weight_2 \leq \weight_3 \leq \ldots \leq \weight_n$ holds, this is the only summand where the value $\weight_{\max}$ is taken. 
Let $\weight_{\max-1}<\weight_{\max}$ denote the second highest weight for the $Z\GL(n)$ which must have the form 
$$\weight_{\max - 1} = \weight_1 + \weight_2 + \cdots + 2\weight_{i}+\weight_{i+2}+\cdots + \weight_n= \weight_{\max} - \weight_{i+1} + \weight_i$$
for some $1\le i \le n-1$. 

Let $\chi: \hU \to \CC^*$ be a character of $\hU$. We will want to choose $p$ and $\chi$ such that
$$ p (\weight_{\max} - n) (\weight_1 + \cdots + \weight_n) - \chi > 0 > p (\weight_{\max - 1} - n) (\weight_1 + \cdots + \weight_n) - \chi $$
or equivalently
\begin{equation} \label{welladapted}
\weight_{\max-1}-n < \frac{\chi}{p(\weight_1 + \cdots \weight_n)} < \weight_1 + \cdots + \weight_n - n. \end{equation}
We call rational characters $\chi/p$ with this property {\it well-adapted}. The linearisation of the action of $\hat{U}$ on $X$ with respect to $L^{\otimes p}$ can be twisted by $\chi$ so that the weights $\rho_j$ of $Z\GL(n)$ are replaced with $\rho_j p-\chi$ for $j=0,\ldots, s$. Let $L_\chi^{\otimes p}$ denote this twisted linearisation. 
 
\subsection{Hilbert-Mumford for the left action of $\SL(n)$}

Recall that $\SL(n) = \SU(n) B_{\SL(n)}$ where $B_{\SL(n)}$ is the standard (upper triangular) Borel subgroup of $\SL(n)$ and $\SU(n)$ is compact, so that
$$\overline{\GL(n)/\hU} = \overline{\SL(n)[\zdis_n] } = \SU(n) ( \overline{B_{\SL(n)}[\zdis_n]} ) .$$
Moreover 
 $\PP(\mathcal{W}_{\det})$
is $\SL(n)$-invariant, so 
$$ \PP(\mathcal{W}_{\det}) \cap \overline{\GL(n)/\hU} =  \SU(n) ( \PP(\mathcal{W}_{\det}) \cap \overline{B_{\SL(n)}[\zdis_n]} ) .$$
Now fix positive integers $\rho_1 \gg \rho_2 \gg \ldots \gg \rho_{n-1}>0$ and consider the left action of the one-parameter subgroup $\CC^*_\rho \leq \SL(n)$ given by 
\[ t \mapsto  \left(\begin{array}{ccccc} t^{\rho_1} & & & &\\ & t^{\rho_2} & & & \\ & & \ddots & & \\ & & & t^{\rho_{n-1}} & \\
& & & & t^{-(\rho_1 + \rho_2 + \cdots + \rho_{n-1})} \end{array}\right) \mbox{ for $t\in \CC^*$ }. \]
The weights of $\CC^*_\rho$ acting on $\PP(\mathcal{W}_{\det}) \cap \overline{B_{\SL(n)}[\zdis_n]}$ are all of the form $k_1 \rho_1 + k_2 \rho_2 + \cdots + k_{n-1} \rho_{n-1}$
where $k_1, \ldots , k_{n-1}\geq 0$. 

By Remark \ref{rmk3.2} and Proposition \ref{homogprop} $\overline{B_{\SL(n)}[\zdis_n]}$ is contained  in the subspace 
\[\PP^*=\PP(W_1 \wedge \ldots \wedge W_n) \subset \PP(\wsymk n)\]
where the subspaces
\[W_i=\mathrm{Span}_\CC(e_\tau: \mathrm{supp}(\tau)\subseteq \{1,\ldots i\}, \sum_{t\in \tau}\weight_t \le \weight_i)\subset \sym^\weight \CC^n\]  
are invariant under the upper Borel subgroup $B_n \subset \GL(n)$ which preserves the flag $\mathrm{Span}(e_1)\subset \mathrm{Span}(e_1,e_2) \ldots \subset \mathrm{Span}(e_1,\ldots, e_n)$. Here $\tau=(\tau_1\le \tau_2 \le \ldots \le \tau_r)$ is a sequence whose support $\mathrm{supp}(\tau)$ is the set of elements in $\tau$ and $e_{\tau}=\prod_{j\in \tau}e_j=\prod_{i=1}^r e_{\tau_i} \in \sym^{r}\CC^n$.
Basis elements of $\PP^*$ are parametrised by {\it admissible} sequences of partitions $\mathbf{\pi}=(\pi_1,\ldots, \pi_n)$. We call a sequence of partitions $\mathbf{\pi}=(\pi_1 \ldots \pi_n)\in \Pi^{\times n}$ admissible if
\begin{enumerate}
\item $\mathrm{supp}(\pi_l)\subseteq \{1,\ldots l\}$
\item $\sum_{t\in \pi_l}\weight_t \le \weight_l$ for $1\le l \le n$, and 
\item $\pi_l\neq\pi_m$ for $1\leq l\neq m\leq n$. 
\end{enumerate}
We will denote the set of admissible sequences of length $n$ by $\mathbf{\Pi}$. The corresponding basis element  is then $e_{\pi_1} \wedge \ldots \wedge e_{\pi_n} \in W_1 \wedge \ldots \wedge W_n$.

\begin{lemma} For $\rho=(\rho_1 \gg \rho_2 \gg \ldots \gg \rho_{n-1}>0)$ and $\mathbf{\pi} \in \mathbf{\Pi}$ the weight of the left $\CC^*_{\rho}$ action on $e_\pi$ is strictly positive unless $\mathbf{\pi}=(1,2,\ldots, n)$ corresponding to the basis element $e_{(1,2,\ldots, n)}=e_1\wedge \ldots \wedge e_n$.
\end{lemma}

\begin{proof}
The weight $\rho_\pi$ of the left $\CC^*_{\rho}$ action on $e_{\mathrm{\pi}}=e_{\pi_1}\wedge \ldots \wedge e_{\pi_n}$ is 
\[\rho_{\mathrm{\pi}}=\sum_{i=1}^n \sum_{j \in \pi_l} \rho_j \ge \sum_{i=1}^{n-1} \rho_i+\sum_{j \in \pi_n} \rho_j\]
whenever $\rho_1 \gg \rho_2 \gg \ldots \gg \rho_{n-1}>0$ holds by (1) and (2) in the definition of admissible sequences. Moreover, equality holds if and only if $\pi_l=(l)$ for $1\le l \le n-1$. Finally $\sum_{j \in \pi_n} \rho_j\ge -(\rho_1+\ldots +\rho_{n-1})$ with equality if and only if $\pi_n=(n)$, otherwise $e_n$ does not appear in $e_{\pi_n}$. 
\end{proof}

Let $\eta_{\min} = \eta_1 < \cdots < \eta_\rho = \eta_{\max}$
be the weights of the action of $\CC^*_\rho$ on $X$ with respect to the linearisation $L$ of the $\SL(n)$ action. Then if $q\eta_{\min} + p n > 0$ it follows that every point of $(\PP(\mathcal{W}_{\det})\times X) \cap \overline{B_{\SL(n)}[\zdis_n] \times X} )$ is unstable for the left action of this one parameter subgroup of $\SL(n)$ with respect to the linearisation $L^{(p,q)}$ (or equivalently $L^{(p,q)}_\chi$).  It follows that 

\begin{lemma} \label{lempdet}
If  $p > - q\eta_{\min}/n$ then every point of
\[(\PP(\mathcal{W}_{\det}) \times X) \cap \overline{(\GL(n)/\hU)  \times X} =  \SU(n) ( (\PP(\mathcal{W}_{\det}) \times X) \cap \overline{B_{\SL(n)}[\zdis_n]} \times X )\]
is unstable for the left action of  $\SL(n)$ with respect to the linearisation $L^{(p,q)}$ (or equivalently $L^{(p,q)}_\chi$).
\end{lemma}


\subsection{Hilbert-Mumford for the right action of $\CC^* \leq \tilde{U}$}

Recall 
that the boundary  of $\GL(n) \times_{\hU} X = (\GL(n)/{\hU}) \times X$ in the projective completion \[\overline{\GL(n)/{\hU}} \times X \subset \PP(\wsymk n) \times X\]  is contained in the union of  
$\PP(\mathcal{W}_{v_1}) \times X$ and $\PP(\mathcal{W}_{\det}) \times X$.
If we twist the linear action of $\tilde{U}$ on $X$ with respect to $L^{\otimes p}$ which extends to a linear action of $\SL(n)$ by a character $\chi$, then the induced right action of the one parameter subgroup $\CC^* \leq \tilde{U}$ on $\overline{\GL(n)/{\hU}} \times X$ with respect to the line bundle $\mathcal{O}_{\PP(\wsymk n)}(p) \otimes \mathcal{O}_{\PP^N}(q)$ has weights
$$ p ( k_1 \weight_1 + \cdots + k_n \weight_n - n) (\weight_1 +  \cdots + \weight_n) - \chi.$$ 
If the rational character $\chi/p$ is well adapted in the sense of (\ref{welladapted}) then
 the twisted $\SL(n) \times \CC^*$-linearisation  $L^{(p,q)}_{\chi}$ on $\overline{\GL(n)/\hU} \times X$ has strictly negative weights under the right action of the one parameter subgroup $\CC^* \leq \tilde{U}$ on $\PP(\mathcal{W}_{v_1}) \times X$ and therefore all points of $\PP(\mathcal{W}_{v_1}) \times X$ are unstable with respect to this linear action of 
$\SL(n) \times \CC^*$.
We know from Corollary \ref{corbound}
that the boundary of the orbit $\GL(n)[\zdis_n]$ in $\PP(\mathcal{W})$ is contained in the union of the subspaces $\PP(\mathcal{W}_{v_1})$ and $\PP(\mathcal{W}_{\det})$. So combining this with Lemma \ref{lempdet} we obtain

\begin{prop} \label{propbound}
If $p >> q > 0$ and the rational character $\chi/p$ is well adapted in the sense of (\ref{welladapted}), then the boundary of the closure $\overline{\GL(n)/\times_{\hU} X} \cong \overline{\GL(n)/\hU } \times X$ of $\GL(n) \times_{\hU} X$ in
$  \PP(\wsymk n) \times X $ is unstable for the linear action of  $\SL(n) \times \CC^* = \SL(n) \times (\tilde{U}/U)$ with respect to the linearisation $L^{(p,q)}_\chi$.
\end{prop}

Recall that $\GL(n)/\hU = \SL(n) /(\hU \cap \SL(n))$ where $\hU \cap \SL(n)$ is a finite extension of $U$ which is contained in $\tilde{U}$ with $\tilde{U}/(\hat{U} \cap \SL(n)) \cong \CC^*$. It thus follows immediately from Theorem \ref{thm:geomcor}  that we have

\begin{theorem} \label{thmbound}
Let $X$ be a projective variety acted on linearly by $\tilde{U}$, and suppose that the action extends to a linear action of $\SL(n)$  with respect to an ample linearisation. If the linearisation of the $\tilde{U}$-action is twisted by a well-adapted rational character $\chi/p$ for sufficiently divisible $p$, then
\begin{enumerate}
\item the algebra of invariants 
$\bigoplus_{k \geq 0} H^0( X,L^{\otimes kp})^{\tilde{U}}$ is finitely generated;
\item the enveloping quotient $X/\!/\tilde{U} \simeq (\overline{\GL(n)/\hU} \times X) /\!/_{L^{(p,1)}_\chi} (\SL(n)\times \CC^*)\simeq \mathrm{Proj}(\oplus_{k \geq 0} H^0( X,L^{\otimes kp})^{\tilde{U}})$;
\item  the morphism 
\[ \phi:X^{ss,\tilde{U}} \rightarrow X/\!/\tilde{U}\]
is surjective and $X/\!/\tilde{U}$ is a categorical quotient of $X^{ss,\tilde{U}}$ with $\phi(x) = \phi(y)$ if and only if the closures of the $\tilde{U}$-orbits of $x$ and $y$ meet in $X^{ss,\tilde{U}}$.
\end{enumerate} 
\end{theorem}

\subsection{The action of $\tilde{U}$ on $X \times \PP^1$}

Now let us consider the diagonal action of $\tilde{U}$ on $X \times \PP^1$ where $\tilde{U}$ acts on $\PP^1$ linearly with respect to $\calo_{\PP^1}(1)$ by
$$ \tilde{u} [ x_0:x_1] = [\chi_0(\tilde{u})x_0: x_1]$$
where $\chi_0: \tilde{U} \to \CC^*$ is the character with kernel $U$ given by 
\[ \chi_0 \left(\begin{array}{cccc} t^{n\weight_1 -(\weight_1+ \cdot + \weight_n)} & & & \\ & t^{n\weight_2-(\weight_1+ \cdots + \weight_n)} & & \\ & & \ddots & \\ & & & t^{n\weight_n - (\weight_1+ \cdots + \weight_n)} \end{array}\right) = t. \]
We can adapt the arguments of $\S$5.3 and $\S$5.4 to the induced linear action of $\SL(n) \times \CC^*$ on 
$$ \overline{\GL(n)/\hU} \times X \times \PP^1 \subseteq  \PP(\wsymk n) \times \PP^N \times \PP^1
$$
for the linearisation $L^{(p,q,r)}_\chi$ defined with respect to the line bundle $\mathcal{O}_{\PP(\wsymk n)}(p) \otimes \mathcal{O}_{\PP^N}(q) \otimes \mathcal{O}_{\PP^1}(r)$ where the action of $\tilde{U}$ on $X$ extends to a linear action of $\SL(n)$ but is then twisted by a rational character $\chi/p$. Now we want to choose $\chi, p,q$ and $r$ such that  $\chi/p$ is well adapted in the sense of (\ref{welladapted}), and $p > -q \eta_{\min}/n$ as before, and also $r > - q \eta_{\min}$ in order that all points of 
$\overline{\GL(n)/\hU} \times X \times \{ \infty \}$ will be unstable for the action of $\SL(n) \times \CC^*$. Note that $\eta_{\min} <0$ unless the action of $\tilde{U}$ is trivial, and then these two conditions will be satisfied if $p>>q$ and $r>> q$. So the proofs of Proposition \ref{propbound} and Theorem  \ref{thmbound} give us

\begin{prop} \label{propbound2}
If $p >> q > 0$ and $r>> q$ and the rational character $\chi/p$ is well adapted in the sense of (\ref{welladapted}), then the boundary of the closure $ \overline{\GL(n)/\hU } \times X \times \PP^1$ of $\GL(n) \times_{\hU} (X \times \CC)$ in
$  \PP(\wsymk n) \times X \times \PP^1$ is unstable for the linear action of  $\SL(n) \times \CC^* = \SL(n) \times (\tilde{U}/U)$ with respect to the linearisation $L^{(p,q,r)}_\chi$.
\end{prop}

\begin{defn}
Let $X^{\hat{s},U}$ denote the $U$-invariant open subset of $X$ such that $$\{ [\zdis_n] \} \times X^{\hat{s},U} \times \{[1:1]\} = (\{ [\zdis_n] \} \times X \times \{[1:1]\}) \cap (\PP(\wsymk n) \times X \times \PP^1)^{s,\SL(n) \times \CC^*}$$ where $(\PP(\wsymk n) \times X \times \PP^1)^{s, \SL(n) \times \CC^*}$ denotes the stable subset of $\PP(\wsymk n) \times X \times \PP^1$ with respect to  the linearisation $L^{(p,1,r)}_\chi$.
\end{defn}

\begin{theorem} \label{thmbound2}
Let $X$ be a projective variety acted on linearly by $\tilde{U}$, and suppose that the action extends to a linear action of $\SL(n)$  with respect to an ample linearisation. If the linearisation of the diagonal action of $\tilde{U}$ on $X \times \PP^1$ is twisted by a well-adapted rational character $\chi/p$ for sufficiently divisible $p$, then
\begin{enumerate}
\item the algebra of invariants 
$\bigoplus_{k \geq 0} H^0( X \times \PP^1,L^{\otimes kp})^{\tilde{U}}$ is finitely generated;
\item the enveloping quotient $(X \times \PP^1)/\!/\tilde{U} \simeq (\overline{\GL(n)/\hU} \times X \times \PP^1) /\!/ (\SL(n)\times \CC^*)\simeq \mathrm{Proj}(\oplus_{k \geq 0} H^0( X \times \PP^1,L^{\otimes kp} \otimes \calo_{\PP^1}(r))^{\tilde{U}})$ for $r>> 1$;
\item  there is a surjective $\tilde{U}$-invariant morphism 
\[ \phi:(X \times \CC)^{ss,\tilde{U}} \rightarrow (X \times \PP^1)/\!/\tilde{U}\]
from a $\tilde{U}$-invariant open subset $(X \times \CC)^{ss,\tilde{U}}$ of $X \times \CC$ making $X/\!/\tilde{U}$ is a categorical quotient of $(X\times \CC)^{ss,\tilde{U}}$ with $\phi(x) = \phi(y)$ if and only if the closures of the $\tilde{U}$-orbits of $x$ and $y$ meet in $(X \times \CC)^{ss,\tilde{U}}$;
\item this morphism $\phi$ restricts to a geometric quotient
$X^{\hat{s},U} \to X^{\hat{s},U}/U$ for the action of $U$ on the $U$-invariant open subset $X^{\hat{s},U}$ of $X$.
\end{enumerate} 
\end{theorem}

\section{Some applications}

 Recall that if $U$ is { any} unipotent complex linear algebraic group of dimension $n-1$ which has
an action of $\CC^*$ with all weights strictly positive, then $U$ can be embedded in $\GL({\rm Lie}(U \rtimes \CC^*))$ via its adjoint action on the Lie algebra
${\rm Lie}({U}\rtimes \CC^*)$ as the unipotent radical of a subgroup $\hU$ of the form (\ref{label1}) 
 which is generated along the first row, and as the unipotent radical of the associated subgroup $\tilde{U}$ of $\SL(n)$ (which we are calling the adjoint form of $U$). Here the weights of the action of $\CC^*$ on the Lie algebra of $U$ are $\weight_j - 1$ for $j=2,\ldots, n$. We can apply Theorems \ref{thmbound} and \ref{thmbound2} to this situation, and
 also  in the situation of jet differentials considered in $\S$5. Here $U$ is the unipotent radical $\UU_n$ of the reparametrisation group $\GG_n$, and $\tilde{U}$ is the associated subgroup
 $\tilde{\GG}_n \cong \UU_n \rtimes \CC^*$ of $\SL(n)$ which is isomorphic to $\GG_n$ when $n$ is odd and is a double cover of $\GG_n$ when $n$ is even. 

We will finish this section by describing two examples of algebras of invariants in the case of jet differentials.  

\begin{ex}
 \textbf{Invariant jet differentials of order $2$ in dimension $2$.}
 When $n=2$ then 
 $\UU_2$ (which is the unipotent radical of the standard Borel subgroup of $\SL(2)$) is a Grosshans subgroup of $\SL(2)$. 
As usual let  $\{e_1,e_2\}$ be the standard basis for $\CC^2$, and
consider the group 
\[\GG_2 = \left\{ 
\left(\begin{array}{cc}
\a_1 & \a_2 \\
0 & \a_1^2 
\end{array} \right) : \a_1 \in \CC^*, \a_2\in \CC \right.\}=\CC^* \ltimes \CC^{+}\]
with maximal unipotent $\CC^+$ acting on $\CC^2$ by translation. Then $\symk=\CC^n \oplus \sym^2 \CC^2$ has an induced basis $\{e_1,e_2,e_1^2,e_1e_2,e_2^2\}$. Let $x_{ij}$ denote the standard coordinate functions on $\SL(2)\subset (\CC^2)^* \otimes \CC^2$. Then in the notation of $\S$4
\[\phi_1(x_{11},x_{12},x_{21},x_{22})=(x_{11},x_{21}),\]
and 
\[\phi_2(x_{11},x_{12},x_{21},x_{22})=(x_{11},x_{21})\wedge ((x_{12},x_{22})+ (x_{11}^2,2x_{11}x_{21},x_{21}^2)),\]
and $\calo(\SL(n))^U$ is generated by $x_{11},x_{21}$ and the $2 \times 2$ minors of 
\[\left(\begin{array}{ccccc}
x_{11} & x_{21} & 0 & 0 & 0 \\
x_{12} & x_{22} & x_{11}^2 & 2x_{11}x_{21} &
x_{21}^2
\end{array}\right). 
\]
Since the determinant is $1$ this set of generators reduces to two generators $x_{11},x_{21}$, as expected since
$\SL(2)/\CC^+ \cong \CC^2 \setminus \{ 0 \}$ and its canonical affine completion 
$\SL(2)/\!/\CC^+$ is $\CC^2$.

\end{ex}

\begin{ex} \textbf{Invariant jet differentials of order $3$ in dimension $3$.}  
When $n=3$ the finite generation of  the Demailly-Semple algebra $\calo((J_3)_x)^{\UU_3}$ was proved by Rousseau in \cite{rousseau}. 
Here 
\[\GG_3 = \left\{ 
\left(\begin{array}{ccc}
\a_1 & \a_2 & \a_3\\
0 & \a_1^2 & 2\a_1\a_2 \\
0 & 0 & \a_1^3
\end{array} \right) : \a_1 \in \CC^*, \a_2,\a_3 \in \CC \right.\}=\CC^* \ltimes U\]
while $\symk=\CC^n \oplus \sym^2 \CC^3 \oplus \sym^3 \CC^3$ has basis $\{e_1,e_2,e_3,e_1^2,e_1e_2,\ldots, e_3^3\}$. Let $x_{ij}$ denote the standard coordinate functions on $\SL(3)$. Then in the notation of $\S$4
\begin{multline}
\phi_3(x_{11},\ldots, x_{33})=(x_{11},x_{21},x_{31})\wedge ((x_{12},x_{22},x_{32}))+ (x_{11}^2,2x_{11}x_{21},x_{21}^2,2x_{21}x_{31},2x_{11}x_{31},x_{31}^2)\\
\wedge ((x_{12},x_{22},x_{32})+(2x_{11}x_{12},\ldots, 2x_{13}x_{23})+(x_{11}^3,\ldots, x_{31}^3))
\end{multline}
and $\calo(\SL(3))^U$ is generated by those minors of   
\[\left(\begin{array}{ccccccccccc}
x_{11} & x_{21} & x_{31} & 0 & 0 & \cdots & 0 & 0 & 0 & \cdots  & 0 \\
x_{12} & x_{22} & x_{32} & x_{11}^2 & 2x_{11}x_{21} & \cdots  & x_{33}^2 & 0 & 0 & \cdots & 0 \\
x_{13} & x_{23} & x_{33} & x_{11}x_{12} &
x_{11}x_{22}+x_{12}x_{21} & \cdots & x_{31}x_{32} & x_{11}^3 & x_{11}^2x_{21} & \cdots & x_{31}^3
\end{array}
\right)
\]
whose rows form an initial segment of $\{1,2,3\}$, that is the minors $\Delta_{\bi_1,\ldots \bi_s}$ with rows $1,\ldots, s$ and columns indexed by $\bi_1,\ldots, \bi_s$, where $s=1,2$ or $3$ and $|\bi_j|\le 3$.  
\end{ex}

\end{document}